\documentclass[10pt,a4paper]{article}
\usepackage[utf8]{inputenc}
\usepackage{amsmath}
\usepackage{amsfonts}
\usepackage{amssymb}
\usepackage{amsthm}
\usepackage{color}
\usepackage{enumerate}
\usepackage{marvosym}
\usepackage{hyperref}
\usepackage{comment}
\usepackage{multicol}
\usepackage{slashed}
\usepackage{bbm}
 \newtheorem{thm}{Theorem}[section]

 \newtheorem*{thm*}{Theorem}
 \newtheorem{cor}[thm]{Corollary}
 \newtheorem{lem}[thm]{Lemma}
 \newtheorem{prop}[thm]{Proposition}
 \theoremstyle{definition}
 \newtheorem{defn}[thm]{Definition}
\newtheorem{hyp}[thm]{Hypothesis}
 \newtheorem{nota}[thm]{Notation}
 \theoremstyle{remark}
 \newtheorem{heur}[thm]{Heuristic}
 \newtheorem{rem}[thm]{Remark}

 \numberwithin{equation}{section}
 
\newcommand{\vertiii}[1]{{\left\vert\kern-0.25ex\left\vert\kern-0.25ex\left\vert #1
  \right\vert\kern-0.25ex\right\vert\kern-0.25ex\right\vert}}

\newcommand{\p}{\partial}

\newcommand{\Tr}{\operatorname{Tr}}
\newcommand{\tr}{\operatorname{tr}}
\newcommand{\dom}{\operatorname{dom}}
\newcommand{\dif}{\operatorname{Dif}}

\newcommand{\ind}{\operatorname{ind}}

\newcommand{\Id}{\operatorname{d}}
\newcommand{\di}{\slashed{\p}}

\newcommand{\IN}{\mathbb{N}}
\newcommand{\IR}{\mathbb{R}}
\newcommand{\IC}{\mathbb{C}}
\newcommand{\Ii}{\mathbbm{i}}
\newcommand{\one}{\mathbbm{1}}

\newcommand{\VERT}[1]{{\left\vert\kern-0.25ex\left\vert\kern-0.25ex\left\vert #1 
    \right\vert\kern-0.25ex\right\vert\kern-0.25ex\right\vert}}

\begin{document}
\title{Higher order spectral shift of Euclidean Callias operators}
\author{Oliver F\"{u}rst}
\date{\today}

\maketitle
\begin{abstract}
    We consider Dirac-Schrödinger operators over odd-dimensional Euclidean space. The conditions for the potential are based on those of C. Callias in his famous paper on the corresponding index problem. However, we treat the case where the potential can take values in unbounded operators of a separable Hilbert space, and crucially, we also do not assume that the potential needs to be invertible outside a compact region. Hence, the Dirac-Schrödinger operator is not necessarily Fredholm. In the setup we discuss, it however still admits a related trace formula in terms of the underlying potential.

    In this paper we express the trace formula for these Callias-type operators in terms of higher order spectral shift functions, leading to a functional equation which generalizes a known functional equation found first by A. Pushnitski.

    To the knowledge of the author, this paper presents the first multi-dimensional non-Fredholm extension of the Callias index theorem involving higher order spectral shift functions. More precisely, we also show that under a Lebesgue point condition on the higher order spectral shift function associated to the potential, the Callias-type operator admits a regularized index, even in non-Fredholm settings. This corresponds to a known Witten index result in the one-dimensional case shown by A. Carey et al.    
    
    The regularized index that we introduce is a minor extension of the classical Witten index, and we present an index formula, which generalizes the classical Callias index theorem. As an example, we treat the case of $(d+1)$-massless Dirac-Schr\"odinger operators, for which we calculate the associated higher order spectral shift functions.
\end{abstract}

\section{Introduction}

The main goal of this paper is to generalize the classical Callias index theorem to the non-Fredholm case by regularizing both sides of the index formula through higher order spectral shift functions. Moreover we present a functional equation for the involved spectral shift functions. Before we proceed in more detail, we begin with some historical background.

C. Callias considered in his paper on axial anomalies \cite{Cal} the index problem for a Dirac-Schr\"odinger operator $D$ with matrix-valued potential $A$ in odd dimension $d$, i.e.
\begin{align}
	D=\Ii c\nabla\otimes\one_{\IC^{m}}+\one_{\IC^{r}}\otimes M_{A},
\end{align}
where $M_{A}$ is the multiplication operator in $L^{2}(\IR^{d},\IC^{m})$ induced by the matrix function $\IR^{d}\ni x\mapsto A(x)\in\IC^{m\times m}$, and
\begin{align}
    c\nabla=\sum_{i=1}^{d}c^{i}\p_{x^{i}},
\end{align}
where $(c^{i})_{i=1}^{d}$ are Clifford matrices of rank $r=2^{\frac{d-1}{2}}$, i.e. for $i,j=1,\ldots,d$,
\begin{align}
    c^{i}c^{j}+c^{j}c^{i}=-2\delta_{ij}\one_{\IC^{r}}.
\end{align}
The central result of Callias' paper is the following index theorem (\cite{Cal}[Theorem 2]): Let $A$ be a smooth family of $m\times m$-matrices with $\left|A\left(x\right)\right|\geq C>0$ for all $\left|x\right|\geq c$ for some constants $c$, $C$. Assume that $A$ is asymptotically homogeneous of order $0$ as $x\to\infty$. Then $D$ is Fredholm, and, with $U\left(x\right):=A\left(x\right)\left|A\left(x\right)\right|^{-1}$, $k=\frac{d-1}{2}$,
	\begin{align}
		\ind D=\frac{1}{2 k!}\left(\frac{\Ii}{8\pi}\right)^{k}\lim_{R\to\infty}\int_{S_{R}}\Tr_{\IC^{m}}U\left(\Id U\right)^{\wedge\left(d-1\right)},
	\end{align}
here, $S_{R}$ denotes the $(d-1)$-sphere of radius $R$ around $0$, and $\left(\Id U\right)^{\wedge\left(d-1\right)}$ the $(d-1)$-fold exterior product of (matrix-)one-forms $\Id U$. There are some issues with the original proof (for example, trace-class membership of involved operators), which, however, have been rectified by F. Gesztesy and M. Waurick in \cite{GesWau}, where the authors also generalize some of the conditions.

Since the publication of Callias' paper, a plethora of generalizations have emerged in the Fredholm case, for which we suggest the following incomplete (and subjective) list of papers: \cite{BotSee}, \cite{Ang1}, \cite{BruMos}, \cite{Ang2}, \cite{Rad}, \cite{Bun},\cite{Kuc}, \cite{CarNis}, \cite{Kot}, \cite{BraShi}, \cite{Shi}, \cite{Dun}, \cite{Cec}, \cite{SchSto}.

In any case, some condition of invertibility of the potential outside a compact region is necessary in the Fredholm case. In this paper, we crucially suspend this assumption. In this case, one still retains a trace formula for the semi-group difference,
\begin{align}
    \Tr_{L^{2}(\IR^{d},H)}\tr_{\IC^{r}}\left(e^{-tD^{\ast}D}-e^{-tDD^{\ast}}\right),
\end{align}
where $H$ is a complex separable Hilbert space, which takes the role of $\IC^{m}$. This trace formula, formulated in terms of the underlying potential $A$ (allowed to be unbounded operator-valued), is the central result of \cite{Furst1}, which we shall discuss in more detail below.

In the one-dimensional case $d=1$, where Callias type index theorems correspond to the famous "Index=Spectral Flow"-theorem, there is already a comprehensive body of work (\cite{Push}, \cite{GLMST}, \cite{CGLS}, \cite{CarGesLevNicSukZan} to give a (incomplete) selection of the main results) dealing with the non-Fredholm case. The key to an invariant description beyond comparative trace formulas in terms of resolvents or semigroups is the translation into a functional equation of two spectral shift functions. The first result of this kind is due to A. Pushnitski in \cite{Push}[Theorem 1.1], which we quote and paraphrase here: Let $(A(t))_{t\in\IR}$ be a family of self-adjoint operators in $H$. If $A_{\pm}=\lim_{t\to\pm\infty}A(t)$ exist in an appropriate sense, and
\begin{align}
    \int_{\IR}\left\|A'(t)\right\|_{S^{1}(H)}\Id t<\infty,
\end{align}
then there exist unique\footnote{We omit the normalization procedure here, since spectral shift functions are generally just unique up to a constant.} Krein spectral shift functions of the pairs $(A_{+},A_{-})$ and $(DD^{\ast}, D^{\ast}D)$. The Krein spectral shift function $\xi_{T,S}:\IR\rightarrow\IR$ of a suitable pair of self-adjoint operators $(T,S)$ has the property that for admissible functions $f:\IR\rightarrow\IC$,
\begin{align}
    \Tr(f(T)-f(S))=\int_{\IR}f'(\lambda)\xi_{T,S}(\lambda)\Id\lambda.
\end{align}
The spectral shift functions then satisfy the functional equation
\begin{align}\label{eq:intro:2}
    \xi_{DD^{\ast}, D^{\ast}D}(\lambda)=\frac{1}{\pi}\int_{-\sqrt{\lambda}}^{\sqrt{\lambda}}\xi_{A_{+},A_{-}}(s)\frac{\Id s}{\sqrt{\lambda-s^{2}}},
\end{align}
for almost every $\lambda>0$. In \cite{CGLS}[Theorem 4.3] the authors show that if $\xi_{A_{+},A_{-}}$ admits a left- and right Lebesgue point at $0$, then also $\xi_{DD^{\ast},D^{\ast}D}$ admits a right Lebesgue point at $0$, and
\begin{align}\label{eq:intro:3}
    \xi_{DD^{\ast},D^{\ast}D}(0+)=\frac{\xi_{A_{+},A_{-}}(0-)+\xi_{A_{+},A_{-}}(0+)}{2}.
\end{align}
Moreover, it is shown that the value at the right Lebesgue point $0$, the value $\xi_{DD^{\ast},D^{\ast}D}(0+)$ gives an invariant description of the Witten index $\ind_{W}D$, which was introduced in \cite{GesSim}, independent of the chosen regularization (resolvent, semi-group). Since the Witten index agrees with the Fredholm index if both indices exist, and the spectral flow of a Fredholm family coincides with the spectral shift function of its endpoints (again, if both objects exist), formula (\ref{eq:intro:3}) can be regarded as a non-Fredholm extension of the "Index=Spectral Flow"-theorem.

Since we also have a principal trace formula in the other cases $d\geq 3$ odd\footnote{For $d$ even the trace is always trivial.}, presented in \cite{Furst1}[Theorem 6.5], it is suggestive that it should be the starting point to develop a framework centred around spectral shift functions, allowing us to rewrite the trace formula invariantly through a functional equation like (\ref{eq:intro:2}).

For the sake of this introduction we present the principal trace formula without the precise conditions on the potential $A$, which takes values in the self-adjoint operators of a separable Hilbert space $H$ with common domain $\dom A_{0}$ of a given self-adjoint "model operator" $A_{0}$ (In the one-dimensional case discussed above one simply puts $A_{0}=A_{-}$). Under the assumptions made in \cite{Furst1}[Hypothesis 2.11] on $A=A_{0}+B$ we have for $t>0$,
\begin{align}\label{eq:intro:1}
    &\Tr_{L^2(\IR^{d},H)}\tr_{\IC^{r}}\left(e^{-tD^{\ast}D}-e^{-tDD^{\ast}}\right)\nonumber\\
    =&\frac{2}{d}(4\pi)^{-\frac{d}{2}}t^{\frac{d}{2}}\int_{\IR^{d}}\int_{s\in\Delta_{d-1}}\Tr_{\IC^{r}\otimes H}\left(\prod_{j=0}^{d-1}\left(\Ii c\nabla A_{\phi}\right)(x)e^{-ts_{j}A_{\phi}^{2}(x)}\right)\Id s~\Id x,
\end{align}
where $\phi\in C_{c}^{\infty}(\IR^{d})$, such that $\phi\equiv 1$ near $0$ is arbitrary and $A_{\phi}:=A_{0}+(1-\phi)(A-A_{0})$. $\Delta_{d-1}$ is the $(d-1)$-dimensional simplex in $\IR^{d}$. In particular, the traces are independent of the choice of $\phi$.

Let us make some heuristic observations on the formula (\ref{eq:intro:1}), which also serves as an outline of the strategy of this paper. We first note that the left hand side looks like the usual trace comparison of the semi-groups of $D^{\ast}D$ and $DD^{\ast}$, which one would be tempted to re-express via the classical Krein spectral shift function of the pair. Unfortunately, taking the partial trace $\tr_{\IC^{r}}$ first is essential, since in all interesting cases\footnote{The necessary decay condition on $\nabla A$ to obtain the trace-class property of the semi-group difference (cf. \cite{Furst2}) then would automatically give a trivial trace if (\ref{eq:intro:1}) is rewritten by Stokes theorem to a version, where one takes limits $R\to\infty$ of integrals over $S_{R}(0)$.} $e^{-tD^{\ast}D}-e^{-tDD^{\ast}}$ is not trace-class in the total space $L^2(\IR^{d},\IC^{r}\otimes H)$.
However, if we develop $e^{-tD^{\ast}D}$, respectively $e^{-tDD^{\ast}}$, into a non-commutative Taylor series around the self-adjoint operator $H_{B}=\frac{1}{2}(D^{\ast}D+DD^{\ast})$, it turns out that due to the trace properties of Clifford matrices, the first Taylor summands up to order $(d-1)$ will be annihilated by $\tr_{\IC^{r}}$, and only the $d$-th Taylor remainders will survive. This opens the door to the framework of higher order spectral shift functions, which are precisely the densities of functionals
\begin{align}
    f^{(d)}\mapsto\Tr\left(f(T+S)-\sum_{j=0}^{d-1}\frac{1}{j!}\p_{s=0}^{j}f(T+sS)\right),
\end{align}
for appropriate functions $f$ on $\IR$, and self-adjoint operators $T,S$. The existence of these kinds of spectral shift functions were conjectured by L. Koplienko in \cite{Kop}, which has been established by D. Potapov, A. Skripka and F. Sukochev in \cite{PotSkrSuk}. The existence result has been extended to the case of relative Schatten-von Neumann operators $S$ with respect to $T$ in \cite{NulSkr}. Of course in case $d=1$ one retains the classical Krein spectral shift function. It turns out that if we modify the approach in \cite{PotSkrSuk} similar to the modification in \cite{NulSkr}, we are able to construct a spectral shift function $\xi$, such that
\begin{align}
    \Tr_{L^{2}(\IR^{d},H)}\tr_{\IC^{r}}\left(f(D^{\ast}D)-f(DD^{\ast})\right)=\int_{0}^{\infty}f^{(d)}(\lambda)\xi(\lambda)\Id\lambda,
\end{align}
holds for appropriate functions $f$. Its precise formulation is the content of Proposition \ref{prop:leftspectralshift} and Theorem \ref{thm:pushfunctional}.

Let us proceed with the heuristic consideration of the right hand side of (\ref{eq:intro:1}). Disregarding the integral over $x\in\IR^{d}$ for the moment, we note that by applying the trace the integrand can be re-expressed through a multiple operator integral with the $d$-fold divided difference of the function $x\mapsto e^{-tx}$ as a symbol, and the $(d+1)$-fold copy of the spectral measures of $A_{\phi}^{2}(x)$ as spectral measures. This multiple operator integral is then applied to the vector of $d$ copies of the operator $\Ii c\nabla A_{\phi}$ and the trace is taken. Such a type of multiple operator integral is again closely related to higher order spectral shift functions, and it will turn out that we will be able to construct a spectral shift function $\eta$, such that
\begin{align}
    \int_{\IR^{d}}\Tr_{\IC^{r}\otimes H}\mathcal{J}\left(\dif_{d}f,A_{\phi}^{2}(x),\Ii c\nabla A_{\phi}(x)\right)\Id x=\int_{0}^{\infty}f^{(d)}(\lambda)\eta(\lambda)\Id\lambda,
\end{align}
where $\mathcal{J}\left(\dif_{d}f,A_{\phi}^{2}(x),\Ii c\nabla A_{\phi}(x)\right)$ denotes the multiple operator integral with the $d$-fold divided difference of a suitable function $f$ as a symbol and $d+1$ copies of the spectral measure of $A_{\phi}^{2}(x)$ as spectral measures, which is applied to the vector of $d$ copies of the operator $\Ii c\nabla A_{\phi}(x)$. The precise statement is the content of Proposition \ref{prop:spectralcallias}, and in the next chapter we shall give a short introduction to the notion of multiple operator integrals more generally.

If we now rewrite both sides of (\ref{eq:intro:1}) in terms of the $d$-order spectral shift functions $\xi$, respectively $\eta$, we will essentially obtain a functional equation for the Laplace transforms $\mathcal{L}(\xi)$, and $\mathcal{L}(\eta)$, which after inversion yields the first principal result of this paper, the functional equation
\begin{align}\label{eq:intro:4}
    \xi(\lambda)=-\frac{\left(\frac{d-1}{2}\right)!}{\pi^{\frac{d+1}{2}}(d-1)!}\int_{0}^{\lambda}(\lambda-\mu)^{\frac{d}{2}-1}\eta(\mu)\Id\mu.
\end{align}
The precise statement is given in Theorem \ref{thm:pushfunctional}. In the last chapter we will show that this formula is compatible with Pushnitski's formula (\ref{eq:intro:2}) in the one-dimensional case, if one notes that the spectral shift function $\eta$ we introduced is in case $d=1$ a "symmetrization" of the Krein spectral shift function $\xi_{A_{+},A_{-}}$:
\begin{align}
    \eta(\lambda)=\frac{\xi_{A_{+},A_{-}}(\sqrt{\lambda})+\xi_{A_{+},A_{-}}(-\sqrt{\lambda})}{2\sqrt{\lambda}},\lambda>0.
\end{align}
Equation \ref{eq:intro:4} can therefore be regarded as the extension of Pushnitski's formula to Euclidean Callias type operators of arbitrary odd dimension.

It is straightforward to see from (\ref{eq:intro:4}) that $\xi$ is in general $\frac{d-1}{2}$-times weakly differentiable, which is the content of Corollary \ref{cor:spectralshiftregular}. However, if we assume that $\xi$ is even $(d-1)$-times weakly differentiable, we may define the notion of a regularized index in the style of the classical Witten index. We call it \textit{partial Witten index}, since it is a minor modification of the classical definition induced by the presence of the partial trace $\tr_{\IC^{r}}$. We quote from Definition \ref{defn:partialwitten}:

\begin{defn}
    Assume Hypothesis\footnote{A collection of conditions which ensures the existence of $\xi$.} \ref{hyp:smoothhyp}. Assume that the spectral shift function $\xi$ is $(d-1)$-times weakly differentiable with $\xi^{(d-1)}\in e^{tX}L^{1}(\IR)$ for some $t>0$. Furthermore assume that $\xi^{(d-1)}$ admits a right Lebesgue point at $0$. Then define the \textit{partial Witten index}
    \begin{align}
        \ind_{W}D_{B}:=-\xi^{(d-1)}(0+).
    \end{align}
\end{defn}

It is not difficult to show that if $D_{B}$ is also Fredholm, then $\ind_{W}D_{B}=\ind D_{B}$, the usual Fredholm index of $D_{B}$, which is the content of Lemma \ref{lem:fredholmwitten}.

It remains to be discussed under which conditions on the underlying potential $A$ the operator $D_{B}$ admits a partial Witten index, which would provide the appropriate replacement for "$A$ needs to be invertible outside a compact region" in the classical Fredholm case to our regularized setup. In this case one would like to have an index formula in terms of the spectral shift function $\eta$. This is the content of the second main result of this paper given in Theorem \ref{thm:indexformula}:

\begin{thm}
    Assume Hypothesis \ref{hyp:smoothhyp}. Assume that the spectral shift function $\eta$ is $\frac{d-1}{2}$-times weakly differentiable with $\eta^{(j)}(0)=0$ for all $0\leq j\leq\frac{d-3}{2}$. Then $\xi$ is $(d-1)$-times weakly differentiable such that for a.e. $\lambda>0$,
    \begin{align}
        \xi^{(d-1)}(\lambda)=-\frac{1}{\pi}(4\pi)^{-\frac{d-1}{2}}\int_{0}^{\lambda}(\lambda-\mu)^{-\frac{1}{2}}\eta^{(\frac{d-1}{2})}(\mu)\Id\mu.
    \end{align}
    If additionally $\eta^{(\frac{d-1}{2})}\in e^{tX}L^{1}(\IR)$ for some $t>0$, and the function $u\mapsto u\eta^{(\frac{d-1}{2})}(u^{2})$ admits a right Lebesgue point at $0$ with value $L$, i.e.
    \begin{align}
        \lim_{h\searrow 0}\int_{0}^{h}\left(u\eta^{(\frac{d-1}{2})}(u^{2})-L\right)\Id u=0,
    \end{align}
    then also $\xi^{(d-1)}$ admits a right Lebesgue point at $0$, and
    \begin{align}
        \ind_W D_{B}=-\xi^{(d-1)}(0+)=(4\pi)^{-\frac{d-1}{2}}L.
    \end{align}
\end{thm}

We close the paper with the inclusion of a non-trivial example, where $H=L^2(\IR,G)$, and $A(x)=A_{0}+B(x)=\Ii\p_{y}+M_{V(x,\cdot)}$ is itself a family of first order differential operators on $\IR$ induced by a potential $V:\IR^{d}\times\IR\rightarrow B_{sa}(G)$. The associated Callias type operator $D_{B}$ is then called "$(d+1)$-massless Dirac-Schr\"odinger operator" and a partial Witten index formula in this case has already been found in \cite{Furst3}. Here, we extend this result by giving a formula for both spectral shift functions $\xi$ and $\eta$. Interestingly both $\xi$ and $\eta$ are given as integral kernel transforms of an "index density function\footnote{The name is chosen due to the fact that if it is integrated over $\IR^{d}$ it gives the partial Witten index of $D_{B}$ up to a dimensional constant.}" $\IR^{d}\ni z\mapsto\ind_{V}(z)\in\IC$.

 \begin{align}
            \eta_{d,A_{0},B}(\mu)=&\int_{\IR^{d}}\Omega_{d}(\mu,z)\ind_{V}(z)\Id z,\\
            \xi_{d,A_{0},B}(\lambda)=&\int_{\IR^{d}}\Sigma_{d}(\lambda,z)\ind_{V}(z)\Id z,
\end{align}
where the integral kernels $\Omega_{d},\Sigma_{d}$ are given explicitly in terms of special functions, and the index density function $\ind_{V}$ only depends on the potential $V$. This last result is the content of Theorem \ref{thm:concreteexample}.

\section{Construction of higher order spectral shift}

Throughout the paper let $H$ be a complex separable Hilbert space, and $S^{p}$ the space of Schatten-von Neumann operators in $H$.

Higher order spectral shift functions $\xi_{n,T,S}$ are perturbation theoretic objects associated with traces of non-commutative Taylor series of a function $f$ developed around a (self-adjoint) point $T$ with perturbation $S$ (also self-adjoint and in Schatten-von Neumann class $S^{n}(H)$. This means for $f$ regular enough (for example $n$-times differentiable with all derivatives admitting integrable Fourier transforms), one has
\begin{align}
    \Tr\left(f(T+S)-\sum_{k=0}^{n-1}\frac{1}{k!}\p_{s=0}^{k}f(T+sS)\right)=\int_{\IR}f^{(n)}(\lambda)\xi_{n,T,S}(\lambda)\Id\lambda,
\end{align}
The existence of $\xi_{n,T,S}\in L^{1}(\IR)$ has been prominently proven in \cite{PotSkrSuk} after a conjecture by L. Koplienko in \cite{Kop}. The main tool in their proof of existence is a strong estimate on a certain type of symmetric multiple operator integrals, which we will cite below. As we shall see, this estimate is strong enough so that we can apply it in the context of this paper and construct the variant of spectral shift we need.

Before we proceed, let us first define multiple operator integrals of order $n$.

The development of the concept started with the series of papers \cite{BirSol1},\cite{BirSol2} and \cite{BirSol3} by M.S. Birman and M.Z. Solomyak in the late 1960's and early 1970's. V.V. Peller extended the theory extensively, and we recommend his survey article (and the enclosed literature list therein) \cite{Pel2} as a starting point to study multiple operator integrals. 

The main idea of multiple operator integrals is to provide a functional calculus for several (in general non-commutative) spectral measures, where the symbol space (i.e. the multiple variable functions which are amenable) remains commutative, by assigning a fixed order to the application of the spectral measures. What we mean by that precisely is the content of the next definition. We broadly follow the framework presented in \cite{Pel1}.

\begin{defn}\label{defn:symbols}
    Let $\mathbf{E}=\left((M_{0},\mathcal{B}_{0},E_{0}),\ldots,(M_{n},\mathcal{B}_{n},E_{n})\right)$ be a vector of spectral measures on $H$, and $\mathbf{T}=(T_{1},\ldots,T_{n})$ a vector of linear operators in $H$. Let $\mathcal{D}\subseteq H$ be a dense subspace. A measurable function $\phi:M_{0}\times\ldots\times M_{n}\rightarrow\IC$ is an element in $\mathcal{S}(\mathbf{E},\mathbf{T};\psi)$, the \textit{symbol space} of the multiple operator integral $\mathcal{J}(\cdot,\mathbf{E},\mathbf{T})$, if for $\psi\in\mathcal{D}$ there exists a measure space $(X,\tau)$, and measurable functions $\alpha_{i}:M_{i}\times X\rightarrow\IC$, such that for a.e. $\lambda_{i}\in M_{i}$
    \begin{align}
        x\mapsto\alpha_{0}(\lambda_{0},x)\cdot\ldots\cdot\alpha_{n}(\lambda_{n},x),
    \end{align}
    is integrable with
    \begin{align}\label{eq:defn:symbols:0}
        \phi(\lambda_{0},\ldots,\lambda_{n})=\int_{X}\alpha_{0}(\lambda_{0},x)\cdot\ldots\cdot\alpha_{n}(\lambda_{n},x)\Id\tau(x),
    \end{align}
    and such that for a.e. $x\in X$ the space $\mathcal{D}$ is contained in the domain of
    \begin{align}\label{eq:defn:symbols:1}
        \left(\int_{M_{0}}\alpha_{0}(\lambda_{0},x)\Id E_{0}(\lambda_{0})\right)\prod_{k=1}^{n}\left(T_{k}\int_{M_{k}}\alpha_{k}(\lambda_{k},x)\Id E_{k}(\lambda_{k})\right),
    \end{align}
    and,
    \begin{align}\label{eq:defn:symbols:2}
        x\mapsto\left(\int_{M_{0}}\alpha_{0}(\lambda_{0},x)\Id E_{0}(\lambda_{0})\right)\prod_{k=1}^{n}\left(T_{k}\int_{M_{k}}\alpha_{k}(\lambda_{k},x)\Id E_{k}(\lambda_{k})\right)\psi,
    \end{align}
      is a function in $L^{1}\left(X,\tau,H\right)$, the space of $H$-Bochner integrable functions.

    For $\psi\in\mathcal{D}$, and $\phi\in\mathcal{S}(\mathbf{E},\mathbf{T};\psi)$ define the \textit{multiple operator integral},\\
    $\mathcal{J}(\phi,\mathbf{E},\mathbf{T})\psi$, as the $H$-Bochner integral
    \begin{align}
        &\mathcal{J}(\phi,\mathbf{E},\mathbf{T})\psi\nonumber\\
        :=&\int_{X}\left(\int_{M_{0}}\alpha_{0}(\lambda_{0},x)\Id E_{0}(\lambda_{0})\right)\prod_{k=1}^{n}\left(T_{k}\int_{M_{k}}\alpha_{k}(\lambda_{k},x)\Id E_{k}(\lambda_{k})\right)\psi\Id\tau(x).
    \end{align}
\end{defn}

\begin{nota}
    For a self-adjoint operator $A$ in $H$ and a linear operator $B$, we use the short hand notation
    \begin{align}
        \mathcal{S}(A,B,\psi):=\mathcal{S}\left((E_{A},\ldots,E_{A}),(B,\ldots,B),\psi\right),\\
        \mathcal{J}(\cdot,A,B):=\mathcal{J}\left(\cdot,(E_{A},\ldots,E_{A}),(B,\ldots,B)\right),
    \end{align}
    where $E_{A}$ is the spectral measure on $\IR$ of $A$.
\end{nota}

By the operatic H\"older inequality it is clear, that if $\phi$ satisfies (\ref{eq:defn:symbols:0}) with
\begin{align}\label{eq:chaper2:1}
    \int_{T}\prod_{j=0}^{n}\left\|\alpha_{j}(\cdot,x)\right\|_{L^{\infty}(M_{j})}\Id\tau(x)=C<\infty,
\end{align}
and $T_{j}\in S^{p_{j}}$, for $p_{j}\in[1,\infty]$, for $j=1,\ldots,n$ with convention $S^{\infty}=B(H)$, which satisfy $r^{-1}=\sum_{j=1}^{n}p_{j}^{-1}$ for $r\in[1,\infty]$, then $\phi\in\mathcal{S}(\mathbf{E},\mathbf{T};\psi)$ for any $\psi\in H$, and $\mathcal{J}\left(\phi,\mathbf{E},\mathbf{T}\right)$ is an operator in $S^{r}(H)$ for $r^{-1}=\sum_{j=1}^{n}p_{j}^{-1}$, which satisfies
\begin{align}
    \left\|\mathcal{J}\left(\phi,\mathbf{E},\mathbf{T}\right)\right\|_{S^{r}}\leq C\prod_{j=1}^{n}\|T_{j}\|_{S^{p_{j}}}.
\end{align}

Let us now introduce an important subclass of symbols, those given by divided differences.

\begin{defn}
    For $f\in C^{n}(\IR)$, define the \textit{divided differences of $f$}, $\dif_{n}f:\IR^{n+1}\rightarrow\IC$, by
    \begin{align}
    \dif_{0}f&:=f,\nonumber\\
        \dif_{n}f(\lambda)&:=\begin{cases}
            \frac{1}{\lambda_{n-1}-\lambda_{n}}\left(\dif_{n-1}f\left(\lambda^{(n-1)}\right)-\dif_{n-1}f\left(\lambda^{(n)}\right)\right),&\lambda_{n}\neq\lambda_{n-1},\\
            \p_{\lambda_{n-1}}\dif_{n-1}f(\lambda_{0},\ldots,\lambda_{n-1}),&\lambda_{n}=\lambda_{n-1},
        \end{cases}
    \end{align}
    where $\lambda^{(j)}:=\left(\lambda_{0},\ldots,\lambda_{j-1},\lambda_{j+1},\ldots,\lambda_{n}\right)\in\IR^{n}$, for $\lambda\in\IR^{n+1}$, $j=0,\ldots,n$.
\end{defn}

A straightforward induction proof shows the following well known formula for divided differences. For $x,y\in\IR^{n}$ we denote
\begin{align}
        \langle x,y\rangle:=\sum_{j=1}^{m}x_{j}y_{j}.
\end{align}

\begin{lem}[Genochi-Hermite]\label{lem:genochihermite}
    Let $f\in C^{n}(\IR)$. Then
    \begin{align}
        \dif_{n}f(\lambda)=\int_{s\in\Delta_{n}}f^{(n)}\left(\langle s,\lambda\rangle\right)\Id s,
    \end{align}
    where $\Delta_{n}=\left\{s\in[0,1]^{n+1}\left|\sum_{j=0}^{n}s_{j}=1\right.\right\}$ is the $n$-simplex in $\IR^{n+1}$.
    
    In particular $\dif_{n}f$ is symmetric in all its arguments.
\end{lem}

A classical question in the theory of multiple operator integrals is under which "accessible" conditions on $f$ the divided difference $\dif_{n}f$ is a symbol in $\mathcal{S}(\mathbf{E},\mathbf{T};\psi)$. It turns out, that it is surprisingly difficult to answer, and remains in its admittedly vaguely formulated generality unresolved to the knowledge of the author. Again for more details surrounding this question, and multiple operator integrals in general, we point to the survey article \cite{Pel2}. However, for most applications a sufficiently large class of functions $f$ for which the factorization of $\dif_{n}f$ satisfies (\ref{eq:chaper2:1}) is well-known, for example it suffices that the $n$-th derivative of $f$ has an integrable Fourier transform. The proof of this fact is classical, but to keep this paper sufficiently self-contained, we include it in the proof of Proposition \ref{prop:divideddifference}.

We modify these classical spaces by weights, their purpose will become clear as soon as we introduce the kind of operators to which we want their multiple operator integrals to be applied. We should note that a similar approach is used in \cite{NulSkr}, of which we will however deviate sufficiently, such that we need to present our approach in a self-contained manner (except for Theorem \ref{thm:principalestimate}, the main estimate in \cite{PotSkrSuk}).

\begin{defn}
    Let $X$ denote the function $x\mapsto x$, and for $z\geq 0$ let $\langle x\rangle_{z}:=(|x|^{2}+z)^{\frac{1}{2}}$, and $\langle x\rangle:=\langle x\rangle_{1}$. Let $\mathcal{S}(\IR)$ be the Schwartz functions of $\IR$, and $\mathcal{S}'(\IR)$ the space of tempered distributions. Let $\operatorname{FM}(\IR)$ be the Banach space of signed finite Radon measures on $\IR$, equipped with the total variation norm $\|\cdot\|_{\operatorname{TV}(\IR)}$. Denote by $\mathcal{F}$ the Fourier transform on $\IR$.
    
    For $m,n\in\IN_{0}$ denote by $W^{m,n}$ the space of functions $f\in C^{n}(\IR)\cap\mathcal{S}'(\IR)$, such that for $\max(n-m,0)\leq l\leq n$,
    \begin{align}
        \mathcal{F}\left((X+\Ii)^{m-n+l}f^{(l)}\right)\in\operatorname{FM}(\IR).
    \end{align}
    We equip $W^{m,n}$ with the (semi-)norm
    \begin{align}
        \left\|f\right\|_{W^{m,n}}:=\sum_{l=\max(n-m,0)}^{n}\left\|\mathcal{F}\left((X+\Ii)^{m-n+l}f^{(l)}\right)\right\|_{\operatorname{TV}(\IR)}.
    \end{align}
    Denote by $W^{m,n}_{0}$ the subspace of $W^{m,n}$ consisting of functions $f$, such that for $\max(n-m,0)\leq l\leq n$
    \begin{align}
        \mathcal{F}\left((X+\Ii)^{m-n+l}f^{(l)}\right)\in L^{1}(\IR).
    \end{align}
    Note that in particular the space $C_{c}^{\infty}(\IR)$ is contained and dense in all $W^{m,n}_{0}$. We also introduce the weaker semi-norm
    \begin{align}
        \left\|f\right\|_{W^{m,n}_{\infty}}:=\sum_{l=\max(n-m,0)}^{n}\left\|(X+\Ii)^{m-n+l}f^{(l)}\right\|_{L^{\infty}(\IR)}.
    \end{align}
\end{defn}

We now introduce the class of (generally unbounded) operators to which we want to apply the multiple operator integral $\mathcal{J}(\dif_{n}f,A,\cdot)$ for $f\in W^{m,n}$, and $A$ self-adjoint. Their main property is that they well-behave on the scale of Hilbert spaces generated by $A$.

\begin{defn}\label{def:resolutionterms}
    Let $m\in\IN_{0}$, $p\in[1,\infty)$, $\alpha\in[0,\infty)$, and $A$ self-adjoint. Let $S^{m,p,\alpha}_{A}$ be the space of operators $T$ in $H$ such that $\langle A\rangle^{\beta}T\langle A\rangle^{-\beta-\alpha}$ is densely defined and extends to an operator in $S^{p}$ for $\beta\in\IR$ with $|\beta|\leq\max(m-1,0)$. On $S^{m,p,\alpha}_{A}$ define the norm
    \begin{align}
        \|T\|_{S^{m,p,\alpha}_{A}}:=\sup_{\beta\in\IR,~|\beta|\leq\max(m-1,0)}\left\|\overline{\langle A\rangle^{\beta}T\langle A\rangle^{-\beta-\alpha}}\right\|_{S^{p}}.
    \end{align}
    We denote $S^{m,p}_{A}:=S^{m,p,0}_{A}$. Let $\mathbf{T}:=(T_{1},\ldots T_{n})$ be a vector of operators in $S_{A}^{m,\mathbf{p}}:=S^{m,p_{1}}_{A}\times\ldots\times S^{m,p_{n}}_{A}$ with $\mathbf{p}:=(p_{1},\ldots,p_{n})\in[1,\infty)^{n}$. For $k\in\{0,\ldots,\min(m,n)\}$ we introduce the operators
    \begin{align}
        F_{k}^{m,n}(A,\mathbf{T})&:=(-1)^{k}\sum_{\stackrel{\gamma\in\IN_{0}^{k}}{m>\gamma_{1}>\ldots>\gamma_{k}}}\prod_{j=1}^{k}(A+\Ii)^{\gamma_{j}}T_{j}(A+\Ii)^{-\gamma_{j}},~k\geq 1,~m\geq 1,\nonumber\\
         F_{k}^{0,n}(A,\mathbf{T})&:=0,~k\geq 1,\nonumber\\
          F_{0}^{m,n}(A,\mathbf{T})&:=1.
    \end{align}
    For $f\in W^{m,n}$ we introduce
    \begin{align}
        \mathcal{M}_{k}^{m,n}(f,A,\mathbf{T}):=\mathcal{J}\left(\dif_{n-k}\left((X+\Ii)^{m-k}f\right),A,(T_{k+1},\ldots,T_{n})\right).
    \end{align}
\end{defn}

We will utilize the following central estimate from \cite{PotSkrSuk}. It is the "work horse" in the construction of higher order spectral shift measures, which turn out to be absolutely continuous with respect to the Lebesgue measure. We will mimic this procedure in our setup below.

\begin{thm}[Theorem 5.3/Remark 5.4 \cite{PotSkrSuk}]\label{thm:principalestimate}
    Let $p_{j}\in(1,\infty)$, such that $q^{-1}:=\sum_{j=1}^{n}p_{j}^{-1}\in(0,1)$. Then there is a constant $c$, only dependent on $p_{1},\ldots,p_{n}$, such that for $f\in C^{n}$, $A$ self-adjoint, and $B_{j}\in S^{p_{j}}$,
    \begin{align}
        \left\|\mathcal{J}\left(\dif_{n}f,A,(B_{1},\ldots,B_{n})\right)\right\|_{S^{q}}\leq c\|f^{(n)}\|_{\infty}\prod_{j=1}^{n}\|B_{j}\|_{S^{p_{j}}}.
    \end{align}
\end{thm}

We directly gain an estimate for the operators we introduced in Definition \ref{def:resolutionterms}.

\begin{lem}\label{lem:essentialestimates}
    Let $m\in\IN_{0}$, $n\in\IN$ and let $\mathbf{T}\in S^{m,\mathbf{p}}_{A}$ be as in Definition \ref{def:resolutionterms} with $\mathbf{p}\in(1,\infty)^{n}$. Then there is a constant $c$ only dependent on $m,n,\mathbf{p}$, such that for $0\leq k\leq\min(m,n)$, and $f\in W^{m,n}$,
    \begin{align}\label{eq:lem:essentialestimates:1}
        \left\|F_{k}^{m,n}(A,\mathbf{T})\right\|_{S^{q_{k}}}\leq c\prod_{j=1}^{k}\left\|T_{j}\right\|_{S^{m,p_{j}}_{A}},
        \end{align}
        with $q_{k}^{-1}:=\sum_{j=1}^{k}p_{j}^{-1}$, and
        \begin{align}\label{eq:lem:essentialestimates:2}
        \left\|\mathcal{M}_{k}^{m,n}(f,A,\mathbf{T})\right\|_{S^{r_{k}}}&\leq c\left\|\left((X+\Ii)^{m-k}f\right)^{(n-k)}\right\|_{\infty}\prod_{j=k+1}^{n}\left\|T_{j}\right\|_{S^{p_{j}}},
    \end{align}
    with $r_{k}^{-1}:=\sum_{j=k+1}^{n}p_{j}^{-1}$.
\end{lem}

\begin{proof}
    The first estimate (\ref{eq:lem:essentialestimates:1}) is a direct consequence of the H\"older inequality for Schatten-von Neumann operators. The second estimate (\ref{eq:lem:essentialestimates:2}) follows from Theorem \ref{thm:principalestimate}.
\end{proof}

The reason we introduce the operators in Definition \ref{def:resolutionterms} is the following decomposition. It is similar to the approach used in \cite{NulSkr}, where the authors use a similar factorization, albeit without the additional parameter $\alpha\geq 1$, which is essential for our purposes. Their approach leads to a spectral shift function for relative Schatten-von Neumann perturbations, i.e. for operators $T\in S^{0,p,1}_{A}$.

\begin{lem}\label{lem:weighteddif}
    Let $m,n\in\IN_{0}$, and $\mathbf{T}=(T_{1},\ldots,T_{n})$ a vector of bounded operators such that $\langle A\rangle^{\beta}T\langle A\rangle^{-\beta}$ is densely defined and extends to a bounded operator for $\beta\in\IR$ with $|\beta|\leq\max(m-1,0)$. For $f\in W^{m,n}$ we have
    \begin{align}
        (A+\Ii)^{m}\mathcal{J}\left(\dif_{n}f,A,\mathbf{T}\right)=\sum_{k=0}^{\min(m,n)}F^{m,n}_{k}(A,\mathbf{T})\mathcal{M}_{k}^{m,n}(f,A,\mathbf{T}).
    \end{align}
\end{lem}

\begin{proof}
    We perform induction in $n$. For $n=0$ the statement is true. For $n\geq 1$ let $\widetilde{\mathbf{T}}:=(T_{2},\ldots,T_{n})$, then the product rule of divided differences implies
    \begin{align}
        &(A+\Ii)^{m}\mathcal{J}\left(\dif_{n}f,A,\mathbf{T}\right)=\mathcal{J}\left(\dif_{n}\left((X+\Ii)^{m}f\right),A,\mathbf{T}\right)\nonumber\\
        &-\sum_{j=0}^{m-1}(A+\Ii)^{j}T_{1}\mathcal{J}\left(\dif_{n-1}\left((X+\Ii)^{m-j-1}f\right),A,\widetilde{\mathbf{T}}\right)\\
        =&\mathcal{M}_{0}^{m,n}\left(f,A,\mathbf{T}\right)-\sum_{j=0}^{m-1}\left[(A+\Ii)^{j}T_{1}(A+\Ii)^{-j}\right]\nonumber\\
        &(A+\Ii)^{j}\mathcal{J}\left(\dif_{n-1}\left((X+\Ii)^{m-j-1}f\right),A,\widetilde{\mathbf{T}}\right)
    \end{align}
    \begin{align}
        =&\mathcal{M}_{0}^{m,n}\left(f,A,\mathbf{T}\right)-\sum_{j=0}^{m-1}\left[(A+\Ii)^{j}T_{1}(A+\Ii)^{-j}\right]\sum_{k=0}^{\min(j,n-1)}F^{j,n-1}_{k}\left(A,\widetilde{\mathbf{T}}\right)\nonumber\\
        &\mathcal{M}_{k}^{j,n-1}\left((X+\Ii)^{m-j-1}f,A,\widetilde{\mathbf{T}}\right)
    \end{align}
    \begin{align}
        =&F_{0}^{m,n}(A,\mathbf{T})\mathcal{M}_{0}^{m,n}\left(f,A,\mathbf{T}\right)-\sum_{k=0}^{\min(m-1,n-1)}(-1)^{k}F_{k+1}^{m,n}(A,\mathbf{T})\mathcal{M}_{k+1}^{m,n}(f,A,\mathbf{T})\nonumber\\
        =&\sum_{k=0}^{\min(m,n)}F^{m,n}_{k}(A,\mathbf{T})\mathcal{M}_{k}^{m,n}(f,A,\mathbf{T}).
    \end{align}
\end{proof}

Due to the decomposition above we are able to produce the following trace norm estimate for the multiple operator integral applied to a vector of unbounded operators $\mathbf{T}$.

\begin{cor}\label{cor:multiopestimate}
    Let $m\in\IN_{0}$, $n\in\IN$, $\alpha\in[0,\infty)$ such that $m\geq\alpha(n+1)$, and $\mathbf{p}\in(1,\infty)^{n}$. Let $q^{-1}:=\sum_{j=1}^{n}p_{j}^{-1}$. Let $T_{0}\in S^{m,q,\alpha}_{A}$, and $\mathbf{T}\in S^{m,\mathbf{p},\alpha}_{A}$. Then for $f\in W^{m,n}$
    \begin{align}
        \left\|T_{0}\mathcal{J}\left(\dif_{n}f,A,\mathbf{T}\right)\right\|_{S^{1}}\leq c\left\|f\right\|_{W^{m,n}_{\infty}}\left\|T_{0}\right\|_{S^{m,q,\alpha}_{A}}\prod_{j=1}^{n}\left\|T_{j}\right\|_{S^{m,p_{j},\alpha}_{A}},
    \end{align}
    where the constant $c$ is independent of the operators $T_{0},\mathbf{T}, A$ and the function $f$.
\end{cor}

\begin{proof}
     For $1\leq j\leq n$ introduce the operators
    \begin{align}
        S_{j}:=(A+\Ii)^{\beta_{j}}T_{j}(A+\Ii)^{-\beta_{j}-\alpha},
    \end{align}
    with $\beta_{j}:=(j-n-1)\alpha$, and let $S_{0}:=T_{0}(A+\Ii)^{-\alpha}$. Then $S_{0}\in S^{m,q}_{A}$, and $\mathbf{S}:=(S_{1},\ldots,S_{n})\in S^{m,\mathbf{p}}_{A}$. We then have
    \begin{align}
        T_{0}\mathcal{J}(\dif_{n}f,A,\mathbf{T})=S_{0}(A+\Ii)^{\alpha(n+1)-m}(A+\Ii)^{m}\mathcal{J}\left(\dif_{n}f,A,\mathbf{S}\right).
    \end{align}
    We apply Lemma \ref{lem:weighteddif} and get
    \begin{align}\label{eq:cor:multiopestimate:1}
        T_{0}\mathcal{J}(\dif_{n}f,A,\mathbf{T})=S_{0}(A+\Ii)^{\alpha(n+1)-m}\sum_{k=0}^{\min(m,n)}F^{m,n}_{k}(A,\mathbf{S})\mathcal{M}_{k}^{m,n}(f,A,\mathbf{S}).
    \end{align}
    By Lemma \ref{lem:essentialestimates} we estimate each summand and obtain
    \begin{align}
        &\left\|T_{0}\mathcal{J}(\dif_{n}f,A,\mathbf{T})\right\|_{S^{1}}\nonumber\\
        \leq &c\sum_{k=0}^{\min(m,n)}\left\|\left((X+\Ii)^{m-k}f\right)^{(n-k)}\right\|_{\infty}\left\|S_{0}\right\|_{S^{m,q}_{A}}\prod_{j=1}^{n}\left\|S_{j}\right\|_{S^{m,p_{j}}_{A}}\nonumber\\
        =&c\left\|f\right\|_{W^{m,n}_{\infty}}\left\|T_{0}\right\|_{S^{m,q,\alpha}_{A}}\prod_{j=1}^{n}\left\|T_{j}\right\|_{S^{m,p_{j},\alpha}_{A}}.
    \end{align}
\end{proof}

By mimicking the proof idea of \cite{PotSkrSuk}[Theorem 1.1], we construct in a first step towards spectral shift functions the following "higher order spectral shift measures".

\begin{prop}\label{prop:spectralshiftexist}
    Let $m\in\IN_{0}$, $n\in\IN$, $\alpha\in[0,\infty)$ such that $m\geq\alpha(n+1)$, and $\mathbf{p}\in(1,\infty)^{n}$. Let $q^{-1}:=\sum_{j=1}^{n}p_{j}^{-1}$. Let $T_{0}\in S^{m,q,\alpha}_{A}$, and $\mathbf{T}\in S^{m,\mathbf{p},\alpha}_{A}$. Then there exists a Radon measure $\tau_{n,A,T_{0},\mathbf{T}}$ such that $x\mapsto\langle x\rangle^{-m}\Id\tau_{n,A,T_{0},\mathbf{T}}(x)$ is a finite measure on $\IR$,
    and there is a constant $c$ independent of $A$, $T_{0}$, and $\mathbf{T}$, such that
    \begin{align}
        \left\|x\mapsto\langle x\rangle^{-m}\Id\tau_{n,A,T_{0},\mathbf{T}}(x)\right\|_{\operatorname{TV}}\leq c\left\|T_{0}\right\|_{S^{m,q,\alpha}_{A}}\prod_{j=1}^{n}\left\|T_{j}\right\|_{S^{m,p_{j},\alpha}_{A}},
    \end{align}
    and such that for $f\in W^{m,n}_{0}$,
    \begin{align}\label{eq:prop:spectralshifexist:0}
        \Tr\left(T_{0}\mathcal{J}\left(\dif_{n}f,A,\mathbf{T}\right)\right)=\int_{\IR}f^{(n)}(x)\Id\tau_{n,A,T_{0},\mathbf{T}}(x).
    \end{align}
\end{prop}

\begin{proof}
    Consider the decomposition (\ref{eq:cor:multiopestimate:1}) provided in Corollary \ref{cor:multiopestimate}. We introduce the multilinear maps $\psi_{k,A}:W^{m,n}\times S^{m,q,\alpha}_{A}\times S^{m,\mathbf{p},\alpha}_{A}\rightarrow\IC$ given by
    \begin{align}\label{eq:prop:spectralshiftexist:a}
        \psi_{k,A}(f,T_{0},\mathbf{T}):=\Tr\left(S_{0}(A+\Ii)^{\alpha(n+1)-m}F^{m,n}_{k}(A,\mathbf{S})\mathcal{M}_{k}^{m,n}(f,A,\mathbf{S})\right)
    \end{align}
    The map $\psi_{k,A}$ is continuous in the following sense by Lemma \ref{lem:essentialestimates}
    \begin{align}\label{eq:prop:spectralshifexist:1}
        \left|\psi_{k,A}(f,T_{0},\mathbf{T})\right|&\leq c\left\|\left((X+\Ii)^{m-k}f\right)^{(n-k)}\right\|_{\infty}\left\|S_{0}\right\|_{S^{m,q}_{A}}\prod_{j=1}^{n}\left\|S_{j}\right\|_{S^{m,p_{j}}_{A}}\nonumber\\
        &\leq c\left\|\left((X+\Ii)^{m-k}f\right)^{(n-k)}\right\|_{\infty}\left\|T_{0}\right\|_{S^{m,q,\alpha}_{A}}\prod_{j=1}^{n}\left\|T_{j}\right\|_{S^{m,p_{j},\alpha}_{A}}.
    \end{align}
    By the Riesz-Markov representation theorem there exist finite Radon measures $\nu_{k,A,T_{0},\mathbf{T}}$ with $\left\|\nu_{k,A,T_{0},\mathbf{T}}\right\|_{\operatorname{TV}}\leq c\left\|T_{0}\right\|_{S^{m,q,\alpha}_{A}}\prod_{j=1}^{n}\left\|T_{j}\right\|_{S^{m,p_{j},\alpha}_{A}}$, such that
    \begin{align}\label{eq:prop:spectralshiftexist:b}
        \psi_{k,A}(f,T_{0},\mathbf{T})=\int_{\IR}\p_{x}^{n-k}\left((x+\Ii)^{m-k}f(x)\right)\Id\nu_{k,A,T_{0},\mathbf{T}}(x).
    \end{align}
    Thus
    \begin{align}
        \Tr\left(T_{0}\mathcal{J}(\dif_{n}f,A,\mathbf{T})\right)=\sum_{k=0}^{\min(m,n)}\int_{\IR}\p_{x}^{n-k}\left((x+\Ii)^{m-k}f(x)\right)\Id\nu_{k,A,T_{0},\mathbf{T}}(x)
    \end{align}
    If we let $f\in C_{c}^{\infty}(\IR)\subset W^{m,n}$, product rule expansion and integration by parts yields that for $g_{n,A,T_{0},\mathbf{T}}$ defined by
    \begin{align}\label{eq:prop:spectralshiftexist:2}
        g_{n,A,T_{0},\mathbf{T}}(x):=&\sum_{r=1}^{\min(m,n)}\int_{(0,x]}(x-y)^{r-1}(y+\Ii)^{m-r}\nonumber\\\
        &\sum_{k=0}^{r}(-1)^{r}\frac{(n-k)!(m-k)!}{(n-r)!(r-k)!(m-r)!(r-1)!}\Id\nu_{k,A,T_{0},\mathbf{T}}(y),
    \end{align}
    which satisfies $|g_{n,A,T_{0},\mathbf{T}}(x)|\leq c\langle x\rangle^{m-1}\sum_{k=1}^{\min(m,n)}\left\|\nu_{k,A,T_{0},\mathbf{T}}\right\|_{\operatorname{TV}}$, where $c$ is independent of $x$, $A$, $T_{0}$ and $\mathbf{T}$, such that
    \begin{align}
        \Tr\left(T_{0}\mathcal{J}(\dif_{n}f,A,\mathbf{T})\right)=\int_{\IR}f^{(n)}(x)\left(g_{n,A,T_{0},\mathbf{T}}(x)\Id x+(x+\Ii)^{m}\Id\nu_{0,A,T_{0},\mathbf{T}}(x)\right).
    \end{align}
   We define the Radon measure
   \begin{align}
       \Id\tau_{n,A,T_{0},\mathbf{T}}(x):=g_{n,A,T_{0},\mathbf{T}}(x)\Id x+(x+\Ii)^{m}\Id\nu_{0,A,T_{0},\mathbf{T}}(x),
   \end{align}
   which satisfies that $x\mapsto \langle x\rangle^{-m}\Id\tau_{n,A,T_{0},\mathbf{T}}(x)$ is a finite measure on $\IR$ with
    \begin{align}
        &\left\|x\mapsto \langle x\rangle^{-m}\Id\tau_{n,A,T_{0},\mathbf{T}}(x)\right\|_{\mathrm{TV}}\leq c\sum_{k=0}^{\min(m,n)}\left\|\nu_{k,A,T_{0},\mathbf{T}}\right\|_{\operatorname{TV}}\nonumber\\
        &\leq c'\left\|T_{0}\right\|_{S^{m,q,\alpha}_{A}}\prod_{j=1}^{n}\left\|T_{j}\right\|_{S^{m,p_{j},\alpha}_{A}},
    \end{align}
    and such that for $f\in C^{\infty}_{c}(\IR)$,
    \begin{align}
        \Tr\left(T_{0}\mathcal{J}(\dif_{n}f,A,\mathbf{T})\right)=\int_{\IR}f^{(n)}(x)\Id\tau_{n,A,T_{0},\mathbf{T}}(x).
    \end{align}
    The above equation extends by density and continuity to all $f\in W^{m,n}_{0}$.
\end{proof}

\begin{rem}
    We note that for general $m\in\IN_{0}$ large enough, a measure satisfying (\ref{eq:prop:spectralshifexist:0}) from Proposition \ref{prop:spectralshiftexist} is not necessarily unique (it might be altered by $p(x)\Id x$ for any polynomial $p$ of degree less than $n$). However we always have the unique measure given by
    \begin{align}
       \Id\tau_{n,A,T_{0},\mathbf{T}}(x):=g_{n,A,T_{0},\mathbf{T}}(x)\Id x+(x+\Ii)^{m}\Id\nu_{0,A,T_{0},\mathbf{T}}(x),
   \end{align}
   where the function $g_{n,A,T_{0},\mathbf{T}}$ and the Radon measure $\Id\nu_{0,A,T_{0},\mathbf{T}}$ are uniquely given in the proof of Proposition \ref{prop:spectralshiftexist} by (\ref{eq:prop:spectralshiftexist:2}) respectively (\ref{eq:prop:spectralshiftexist:a}) and (\ref{eq:prop:spectralshiftexist:b}).
\end{rem}

\begin{lem}\label{lem:support}
    Under the assumptions of Proposition \ref{prop:spectralshiftexist} let $\tau_{n,A,T_{0},\mathbf{T}}$ be the unique measure given by the formula    
    \begin{align}
       \Id\tau_{n,A,T_{0},\mathbf{T}}(x):=g_{n,A,T_{0},\mathbf{T}}(x)\Id x+(x+\Ii)^{m}\Id\nu_{0,A,T_{0},\mathbf{T}}(x),
   \end{align}
   where the function $g_{n,A,T_{0},\mathbf{T}}$ and the Radon measure $\Id\nu_{0,A,T_{0},\mathbf{T}}$ are uniquely given in the proof of Proposition \ref{prop:spectralshiftexist} by (\ref{eq:prop:spectralshiftexist:2}) respectively (\ref{eq:prop:spectralshiftexist:a}) and (\ref{eq:prop:spectralshiftexist:b}).

   If $A\geq 0$, then $\operatorname{supp}\tau_{n,A,T_{0},\mathbf{T}}\subseteq[0,\infty)$. In this case, the latter support condition also fixes the measure $\tau_{n,A,T_{0},\mathbf{T}}$ uniquely.
\end{lem}

\begin{proof}
    Let $f\in W^{m,n}$ and $0\leq k\leq n$. Assume that $\operatorname{supp}\left((X+\Ii)^{m-k}f\right)^{(n-k)}\subseteq(-\infty,0)$. For $\lambda\in\IR_{\geq 0}^{n-k+1}$ we thus have
    \begin{align}
        \dif_{n-k}\left((X+\Ii)^{m-k}f\right)(\lambda)=\int_{s\in\Delta_{n-k}}\left((X+\Ii)^{m-k}f\right)^{(n-k)}(\langle s,\lambda\rangle)\Id s=0,
    \end{align}
    since $\langle s,\lambda\rangle\geq 0$. Because $A\geq 0$ it follows that
    \begin{align}
        \mathcal{M}^{m,n}_{k}(f,A,\mathbf{S})=0,
    \end{align}
    where $\mathbf{S}$ is the vector of operators given in the proof of Corollary \ref{cor:multiopestimate}. Variation over functions $f$ with the above support condition imply that the measures $\nu_{k,A,T_{0},\mathbf{T}}$, given by (\ref{eq:prop:spectralshiftexist:a}) and (\ref{eq:prop:spectralshiftexist:b}), are supported on $\IR_{\geq 0}$. As a consequence also $g_{n,A,T_{0},\mathbf{T}}$ given by (\ref{eq:prop:spectralshiftexist:2}) and thus $\tau_{n,A,T_{0},\mathbf{T}}$ are supported on $\IR_{\geq 0}$. Since all measures satisfying (\ref{eq:prop:spectralshifexist:0}) of Proposition \ref{prop:spectralshiftexist} differ by a polynomial the condition to be supported on $\IR_{\geq 0}$ fixes the measure uniquely.
\end{proof}

So far we have constructed spectral shift measures subordinated to a fixed set of operators $A$ and $T_{j}$, $j=0,\ldots, n$. For the application in the context of Callias operators we will discuss in the next chapter, it is necessary that the spectral shift measures depend measurably on the given operators $A$ and $(T_{j})_{j=0}^{n}$.

\begin{lem}\label{lem:measureable}
    Let $M\in\IN_{0}$, $n\in\IN$, $\alpha\in[0,\infty)$ such that $m\geq\alpha(n+1)$, let $\mathbf{p}\in(1,\infty)^{n}$, and define $p_{0}^{-1}:=\sum_{j=1}^{n}p_{j}^{-1}$.
    Let $(M,\mathcal{A},\mu)$ be a measure space, and let $A_{0}$ be self-adjoint in $H$. Let $y\mapsto A(y)$ be a function on $M$ with values in the self-adjoint operators of $H$ with $\dom A(y)^{k}=\dom A_{0}^{k}$ for any $1\leq k\leq m$ and $y$-uniformly equivalent graph norms, and for $0\leq j\leq n$ let $y\mapsto T_{j}(y)$ be a function on $M$ with values in $S^{m,p_{j},\alpha}_{A_{0}}$. Assume that for any $1\leq k\leq m$, and $\psi\in\dom A_{0}^{k}$ the map $y\mapsto A(y)\psi$, is (strongly) measurable from $M$ to $\dom A_{0}^{k-1}$. Assume also that for any $0\leq j\leq n$, and $\psi\in H$ the maps
    \begin{align}
        y\mapsto\overline{\langle A_{0}\rangle^{\beta_{j}+\gamma_{j}}T_{j}(y)\langle A_{0}\rangle^{-\beta_{j}-\gamma_{j}-\alpha}}\psi
    \end{align}
    are (strongly) measurable from $M$ to $H$ for $\beta_{0}=\gamma_{0}=0$, and $\gamma_{j}\in\{0,\ldots,\max(m-1,0)\}$, and $\beta_{j}:=(j-n-1)\alpha$ for $1\leq j\leq n$.
    
    Then the map
    \begin{align}
        y\mapsto\tau_{n,A(y),T_{0}(y),\mathbf{T}(y)}
    \end{align}
    is strongly measurable form $M$ to $\langle X\rangle^{m}\operatorname{FM}(\IR)$, where
    \begin{align}
       \Id\tau_{n,A(y),T_{0}(y),\mathbf{T}(y)}(x):=g_{n,A(y),T_{0}(y),\mathbf{T}(y)}(x)\Id x+(x+\Ii)^{m}\Id\nu_{0,A(y),T_{0}(y),\mathbf{T}(y)}(x)
   \end{align}
   is defined as in the proof of Proposition \ref{prop:spectralshiftexist}.
\end{lem}

\begin{proof}
    We first need to check that $\tau_{n,A(y),T_{0}(y),\mathbf{T}(y)}$ is well-defined. Since the domains of $A_{0}^{k}$ and $A(y)^{k}$ for $1\leq k\leq m$ have $y$-uniformly equivalent graph norms it follows that there is a uniform constant $c$ such that
    \begin{align}
        \left\|T_{j}(y)\right\|_{S^{m,p_{j},\alpha}_{A(y)}}\leq c\left\|T_{j}(y)\right\|_{S^{m,p_{j},\alpha}_{A_{0}}},\ j\in\{0,\ldots,n\}.
    \end{align}
    Then Proposition \ref{prop:spectralshiftexist} gives the desired measure. For $N\in\IN$ consider the finite rank projection $K_{N}:=\sum_{k=1}^{N}\langle\cdot,\psi_{k}\rangle\psi_{k}$, where $(\psi_{k})_{k\in\IN}\subset\bigcap_{j=1}^{\infty}\dom A_{0}^{j}$ is some orthonormal basis of $H$. For notational purposes let $K_{\infty}:=\one_{H}$. For $\epsilon>0$, $N\in\IN\cup\{\infty\}$ we introduce the following $L^{1}(\IR)$-function
    \begin{align}
        &\nu^{\epsilon,N}_{k,A(y),T_{0}(y),\mathbf{T}(y)}(x)\nonumber\\
        :=&(2\pi)^{-1}\int_{\IR}e^{-\Ii x\xi-\epsilon\xi^{2}}\int_{s\in\Delta_{n-k}}\Tr\left(K_{N}S_{0}(y)(A(y)+\Ii)^{\alpha(n+1)-m}\right.\nonumber\\
        &\left.F^{m,n}_{k}(A(y),\mathbf{S}(y))\mathcal{J}\left(\lambda\mapsto e^{\Ii\langle s,\lambda\rangle},A(y),(S_{k+1}(y),\ldots,S_{n}(y))\right)\right)\Id s~\Id\xi,
    \end{align}
    where
    \begin{align}
        S_{j}(y):=(A(y)+\Ii)^{\beta_{j}}T_{j}(y)(A(y)+\Ii)^{-\beta_{j}-\alpha},
    \end{align}
    with $\beta_{j}:=(j-n-1)\alpha$, and $S_{0}(y):=T_{0}(y)(A(y)+\Ii)^{-\alpha}$. For $N\in\IN$ the projection $K_{N}$ is of finite rank with range contained in all $\dom A_{0}^{l}$, $l\in\IN$, and so $y\mapsto\nu^{\epsilon,N}_{k,A(y),T_{0}(y),\mathbf{T}(y)}$ is strongly measurable as a map of $M$ to $L^{1}(\IR)$. Moreover for $\epsilon>0$ we have that $\nu^{\epsilon,N}_{k,A(y),T_{0}(y),\mathbf{T}(y)}$ converges in $L^{1}(\IR)$-norm to $\nu^{\epsilon,\infty}_{k,A(y),T_{0}(y),\mathbf{T}(y)}$ as $N\to\infty$, which follows from $K_{N}\xrightarrow{N\to\infty}K_{\infty}$ in strong operator topology and Lemma \ref{lem:schattconv}. Thus $y\mapsto\nu^{\epsilon,\infty}_{k,A(y),T_{0}(y),\mathbf{T}(y)}$ is also strongly measurable. After embedding continuously into the space of finite signed Radon measures $\operatorname{FM}(\IR)$, the norm estimate displayed in (\ref{eq:prop:spectralshifexist:1}) allows us to conclude the convergence
    \begin{align}
        \left\|x\mapsto\left(\nu^{\epsilon,\infty}_{k,A(y),T_{0}(y),\mathbf{T}(y)}(x)\Id x-\Id\nu_{k,A(y),T_{0}(y),\mathbf{T}(y)}(x)\right)\right\|_{\operatorname{TV}}\xrightarrow{\epsilon\searrow 0}0,
    \end{align}
    so $y\mapsto \nu_{k,A(y),T_{0}(y),\mathbf{T}(y)}$ is strongly measurable as a map from $M$ to $\operatorname{FM}(\IR)$. By (\ref{eq:prop:spectralshiftexist:2}) it follows that
    \begin{align}
       \Id\tau_{n,A(y),T_{0}(y),\mathbf{T}(y)}(x):=g_{n,A(y),T_{0}(y),\mathbf{T}(y)}(x)\Id x+(x+\Ii)^{m}\Id\nu_{0,A(y),T_{0}(y),\mathbf{T}(y)}(x)
   \end{align}
    defines a strongly measurable map $y\mapsto\left(x\mapsto\Id\tau_{n,A(y),T_{0}(y),\mathbf{T}(y)}(x)\right)$ of $M$ to $\langle X\rangle^{m}\operatorname{FM}(\IR)$.
\end{proof}

So far we have avoided the topic of operator differentiability and non commutative Taylor expansions on purpose to highlight that higher order spectral shift \textit{measures} can be more broadly defined for multiple operator integrals with the divided difference as symbol. However, we will now deal with the subclass of multiple operator integrals which are applied to a vector of copies of the same operator, which are closely related with the non commutative Taylor expansion. This relation provides the key to obtain higher order spectral shift \textit{functions}, which is the main trick in the proof of \cite{PotSkrSuk}[Theorem 1.1] to bootstrap from measures to (density) functions.

\begin{defn}\label{defn:taylorremainder}
    Let $A$ be self-adjoint in $H$, $B$ symmetric on $\dom A$, such that $A+sB$ is self-adjoint on $\dom A$ for $s\in [0,1]$. Let $f$ be a function on $\IR$, and assume that $s\mapsto f(A+sB)$ is $n$-times differentiable in the strong operator topology on a dense subspace $\mathcal{D}\subseteq H$ at $s=0$.
    Let $T_{n}(f,A,B)$ denote the $n$th non-commutative Taylor summand of $f$ at $A+B$ developed around $A$, i.e.
    \begin{align}
        T_{n}(f,A,B):=\frac{1}{n!}\frac{\Id^{n}}{\Id s^{n}}\left.\left[f(A+sB)\right]\right|_{s=0}.
    \end{align}
    By $R_{n}(f,A,B)$ denote the $n$th non-commutative Taylor series remainder of $f$ at $T+V$ developed around $T$, i.e.
    \begin{align}
        R_{n}(f,A,B):=f(A+B)-\sum_{k=0}^{n-1}T_{k}(f,A,B).
    \end{align}
\end{defn}

Let us fix the type of perturbations $B$ of $A$ we will need in the following hypothesis.

\begin{hyp}\label{hyp:mainhyp}
    Let $m\in\IN_{0}$, $n\in\IN$, $p\in [n,\infty)$, $\alpha\geq 0$, such that $m\geq\alpha n$. Let $A$ be self-adjoint and let $B\in S^{m,p,\alpha}_{A}$ be symmetric with $\dom A^{k}B\supseteq\dom A^{k+1}$ for $k\in\{0,\ldots,m\}$ and assume
    \begin{align}
        \lim_{z\to\infty}\left\|\langle A\rangle^{k}B\langle A\rangle^{-k}\langle A\rangle_{z}^{-1}\right\|_{B(H)}=0.
    \end{align}
\end{hyp}

The conditions on $B$ are chosen such that $B$ behaves uniformly well on the Hilbert scale generated by $A+tB$ for $t\in[0,1]$.

\begin{lem}\label{lem:katorellich}
    Assume Hypothesis \ref{hyp:mainhyp}, and let $t\in[0,1]$. Then $A_{t}:=A+tB$ is self-adjoint on $\dom A$, and the graph norms of $A^{k}$ and $A_{t}^{k}$ are equivalent, uniformly in $t$, for $1\leq k\leq m$. In particular there is a constant $c$ independent of $B$, such that
    \begin{align}
        \sup_{t\in[0,1]}\left\|B\right\|_{S^{m,p,\alpha}_{A_{t}}}\leq c\left\|B\right\|_{S^{m,p,\alpha}_{A}}.
    \end{align}
\end{lem}

\begin{proof}
By the Kato-Rellich theorem's (\cite{Kato}[Theorem 5.4.3]), it is sufficient to show that for $1\leq k\leq m$,
    \begin{align}\label{eq:lem:katorellich}
        \lim_{z\to\infty}\sup_{t\in[0,1]}\left\|\left(A_{t}^{k}-A^{k}\right)\langle A\rangle_{z}^{-k}\right\|_{B(H)}=0.
    \end{align}
    for all $1\leq k\leq m$. We proceed by induction. The case $k=1$ follows directly by Hypothesis \ref{hyp:mainhyp}. In particular $A_{t}$ is self-adjoint on $\dom A$. Now assume that the statement holds for all $1\leq k\leq n-1$. Therefore the graph norms of $A_{t}^{k}$ and $A^{k}$ are uniformly equivalent. Consider
    \begin{align}
        &\left(A_{t}^{n}-A^{n}\right)\langle A\rangle_{z}^{-n}=\sum_{k=0}^{n-1}A_{t}^{k}tBA^{n-1-k}\langle A\rangle_{z}^{-n}\nonumber\\
        =&t\sum_{k=0}^{n-1}\left[A_{t}^{k}\langle A\rangle^{-k}\right]\left[\langle A\rangle^{k}B\langle A\rangle^{-k}\langle A\rangle^{-1}_{z}\right]\left[A^{n-1-k}\langle A\rangle^{k}\langle A\rangle_{z}^{-n+1}\right].
    \end{align}
    The first factor in the sum is uniformly bounded by induction, the second factor tends to $0$ in norm by Hypothesis \ref{hyp:mainhyp}, and the third factor is trivially bounded by functional calculus. Thus (\ref{eq:lem:katorellich}) holds.
\end{proof}

The first part of the following result is classically well-known, namely that the terms in the non commutative Taylor expansion can be expressed by multiple operator integrals, we however include the proof for completeness without claiming any originality. The second part gives a norm estimate of these Taylor terms by the $W^{m,n}$-semi norm of $f$ and the $S^{m,p,\alpha}_{A}$-norm of the perturbation $B$.

\begin{prop}\label{prop:divideddifference}
Let $f\in \bigcap_{j=0}^{n}W^{0,j}$, let $A$ be self-adjoint, and let $B$ be a linear operator which is continuous from $\dom A^{k}$ to $\dom A^{k-1}$ for any $1\leq k\leq n$. Denote $\mathbf{B}:=(B,\ldots,B)$. Then for any $t\in[0,1]$, and $\psi\in\dom A^{n}$ the function $f(A_{t})\psi$ is $n$-times differentiable in $H$-norm, and for $\psi\in\dom A^{n}$,
    \begin{align}\label{eq:prop:divideddifference:1}
        T_{n}(f,A,B)\psi=\mathcal{J}\left(\dif_{n} f,A,\mathbf{B}\right)\psi,
    \end{align}
    and
    \begin{align}\label{eq:prop:divideddifference:0}
        R_{n}(f,A,B)\psi=n\int_{0}^{1}(1-t)^{n-1}\mathcal{J}\left(\dif_{n} f,A_{t},\mathbf{B}\right)\psi\Id t.
    \end{align}
    If we now assume additionally Hypothesis \ref{hyp:mainhyp}, and $f\in W^{m,n}\cap\bigcap_{j=0}^{n}W^{0,j}$, then also
    \begin{align}
        T_{n}(f,A,B)\in S^{\frac{p}{n}},
    \end{align}
    and
    \begin{align}
        R_{n}(f,A,B)\in S^{\frac{p}{n}}.
    \end{align}
    Moreover we have the estimates
    \begin{align}
        \left\|T_{n}(f,A,B)\right\|_{S^{\frac{p}{n}}}+\left\|R_{n}(f,A,B)\right\|_{S^{\frac{p}{n}}}\leq c\left\|f\right\|_{W^{m,n}}\left\|B\right\|_{S^{m,p,\alpha}_{A}}^{n},
    \end{align}
    where $c$ is a constant independent of $f$ and $B$.
\end{prop}

\begin{proof}
    Let $g_{\xi}:=e^{ix\xi}$. Then Lemma \ref{lem:genochihermite} gives for $\lambda=(\lambda_{0},\ldots,\lambda_{n})$,
    \begin{align}\label{eq:prop:divideddifference:2}
        \dif_{n}g_{\xi}(\lambda)=(i\xi)^{n}\int_{s\in\Delta_{n}}\prod_{k=0}^{n}g_{\xi}(s_{k}\lambda_{k})\Id s.
    \end{align}
    In particular the integrand of (\ref{eq:prop:divideddifference:2}) is bounded point-wise by $|\xi|^{n}$. Since $\widehat{f^{(k)}}=:\mu_{k}\in\operatorname{FM}(\IR)$ for $0\leq k\leq n$, we have $f=(2\pi)^{-1}\int_{\IR}g_{\xi}\Id\mu_{0}(\xi)$ a.e., and $\Id\mu_{k}(\xi)=(i\xi)^{k}\Id\mu_{0}(\xi)$. Thus,
    \begin{align}\label{eq:prop:divideddifference:6}
        \dif_{n}f(\lambda)=(2\pi)^{-1}\int_{\IR}(i\xi)^{n}\int_{s\in\Delta_{n}}\prod_{k=0}^{n}g_{\xi}(s_{k}\lambda_{k})\Id s\ \Id\mu_{0}(\xi),
    \end{align}
    which also shows that for $\psi\in\dom A^{n}=\dom A^{n}_{t}$ we have $\dif_{n}f\in\mathcal{S}(A_{t},B;\psi)$, and
    \begin{align}\label{eq:prop:divideddifference:3}
        \mathcal{J}\left(\dif_{n}f,A_{t},\textbf{B}\right)\psi=(2\pi)^{-1}\int_{\IR}(i\xi)^{n}\int_{s\in\Delta_{n}}g_{\xi}(s_{0}A_{t})\prod_{k=1}^{n}\left(Bg_{\xi}(s_{k}A_{t})\right)\psi\Id s\ \Id\mu_{0}(\xi).
    \end{align}
    On the other hand for $\phi\in H$ we have by Fubini's theorem that for $t'\in\IR$,
    \begin{align}
        f(A_{t+t'})\phi=(2\pi)^{-1}\int_{\IR}g_{\xi}(A_{t+t'})\phi\Id\mu_{0}(\xi),
    \end{align}
    Now we apply Duhamel's formula to $g_{\xi}$,
    \begin{align}
        g_{\xi}(A_{t+t'})-g_{\xi}(A_{t})=it'\xi\int_{0}^{1}g_{\xi}\left((1-s)A_{t}\right)Bg_{\xi}(sA_{t+t'})\Id s,
    \end{align}
    and by induction with $g_{\xi}(s\eta)=g_{s\xi}(\eta)$,
    \begin{align}
        g_{\xi}(A_{t+t'})=&g_{\xi}(A_{t})+\sum_{k=1}^{n-1}(it'\xi)^{k}\int_{s\in\Delta_{k}}\prod_{j=0}^{k-1}\left(g_{\xi}(s_{j}A_{t})B\right)g_{\xi}(s_{k}A_{t})\Id s\nonumber\\
        &+(it'\xi)^{n}\int_{s\in\Delta_{n}}\prod_{j=0}^{n-1}\left(g_{\xi}(s_{j}A_{t})B\right)g_{\xi}(s_{n}A_{t+t'})\Id s,
    \end{align}
    Therefore, keeping in mind that $(i\xi)^{k}\Id\mu_{0}(\xi)$ is a finite measure for $0\leq k\leq n$,
    \begin{align}\label{eq:prop:divideddifference:4}
        &f(A_{t+t'})\psi\nonumber\\
        =&f(A_{t})\psi+(2\pi)^{-1}\sum_{k=1}^{n-1}t'^{k}\int_{\IR}(i\xi)^{k}\int_{s\in\Delta_{k}}\prod_{j=0}^{k-1}\left(g_{\xi}(s_{j}A_{t})B\right)g_{\xi}(s_{k}A_{t})\psi\Id s~\Id\mu_{0}(\xi)\nonumber\\
        &+t'^{n}(2\pi)^{-1}\int_{\IR}(i\xi)^{n}\int_{s\in\Delta_{n}}\prod_{j=0}^{n-1}\left(g_{\xi}(s_{j}A_{t})B\right)g_{\xi}(s_{n}A_{t+t'})\psi\Id s~\Id\mu_{0}(\xi),
    \end{align}
    which implies together with (\ref{eq:prop:divideddifference:3}),
    \begin{align}\label{eq:prop:divideddifference:5}
        &\frac{1}{n!}\frac{\Id^{n}}{\Id t^{n}}f(A+tB)\psi=\frac{1}{n!}\left.\frac{\Id^{n}}{\Id t'^{n}}f(A+tB+t'B)\psi\right|_{t'=0}=T_{n}(f,A_{t},B)\psi\nonumber\\
        =&(2\pi)^{-1}\int_{\IR}(i\xi)^{n}\int_{s\in\Delta_{n}}\prod_{j=0}^{n-1}\left(g_{\xi}(s_{j}A_{t})B\right)g_{\xi}(s_{n}A_{t})\psi\Id s~\Id\mu_{0}(\xi)\nonumber\\
        =&\mathcal{J}\left(\dif_{n}f,A_{t},\textbf{B}\right)\psi.
    \end{align}
    Since $f(A_{t})\psi$ is $n$-times continuously differentiable, we may integrate by parts to obtain
    \begin{align}
        f(A+B)\psi&=\sum_{j=0}^{n-1}T_{j}(f,A,B)\psi+\int_{0}^{1}\frac{(1-t)^{n-1}}{(n-1)!}\frac{\Id^{n}}{\Id t^{n}}f(A_{t})\psi\Id t\nonumber\\
        &=\sum_{j=0}^{n-1}T_{j}(f,A,B)\psi+n\int_{0}^{1}(1-t)^{n-1}\mathcal{J}\left(\dif_{n}f,A_{t},\textbf{B}\right)\psi\Id t.
    \end{align}
    In particular we conclude
    \begin{align}\label{eq:prop:divideddifference:6}
        R_{n}(f,A,B)=n\int_{0}^{1}(1-t)^{n-1}\mathcal{J}\left(\dif_{n}f,A_{t},\textbf{B}\right)\Id t.
    \end{align}
    By Lemma \ref{lem:katorellich} and Theorem \ref{thm:principalestimate} there is a constant independent of $B$ and $f$, such that for all $t\in[0,1]$
    \begin{align}\label{eq:prop:divideddifference:7}
        \left\|\mathcal{J}\left(\dif_{n}f,A_{t},\textbf{B}\right)\right\|_{S^{\frac{p}{n}}}\leq c\left\|f\right\|_{W^{m,n}}\left\|B\right\|_{S^{m,p,\alpha}_{A}}^{n},
    \end{align}
    which immediately implies that $\left\|T_{n}(f,A,B)\right\|_{S^{\frac{p}{n}}}\leq c\left\|f\right\|_{W^{m,n}}\left\|B\right\|_{S^{m,p,\alpha}_{A}}^{n}$, and since the estimate (\ref{eq:prop:divideddifference:7}) is integrable in $t\in[0,1]$, (\ref{eq:prop:divideddifference:6}) converges as a Bochner integral in $S^{\frac{p}{n}}$ with norm bounded by $c'\left\|f\right\|_{W^{m,n}}\left\|B\right\|_{S^{m,p,\alpha}_{A}}^{n}$, which finishes the proof.
\end{proof}

We will also need the following fundamental convergence improving property of Schatten-von Neumann operators.

\begin{lem}\label{lem:schattconv}\cite{GLMST}[Lemma 3.4]
	Let $R_{m},S_{m},T_{m}$, $m\in\IN$ be bounded operators in $H$. For $1\leq p<\infty$, let $S_{m}\in S^{p}$, such that $\lim_{m\to\infty}S_{m}=S$ in $S^{p}$-norm. Furthermore, assume that $\lim_{m\to\infty}R_{m}=R$, and $\lim_{m\to\infty}T_{m}^{\ast}=T^{\ast}$ strongly. Then
	\begin{align}
		\lim_{m\to\infty}R_{m}S_{m}T_{m}=RST,
	\end{align}
	in $S^{p}$-norm.
\end{lem}

We arrive at the principal result of this chapter, the construction of higher order spectral shift functions generated by self-adjoint operators $A$, perturbations $B$ satisfying Hypothesis \ref{hyp:mainhyp}, and an auxiliary operator $T_{0}$. It will be the building block for more concrete spectral shift functions subordinated to Callias type operators, which we shall discuss in the next chapter. We emphasize that the next result is abstract, and not restricted to our intended application. The idea of proof is a modification of the proof of \cite{PotSkrSuk}[Theorem 1.1].

\begin{thm}\label{thm:spectralshiftfunctionexist}
    Assume Hypothesis \ref{hyp:mainhyp} and let $n\geq 2$. Let $T_{0}\in S^{m,\frac{p}{p-n},\alpha}_{A}$, set $\mathbf{B}:=(B,\ldots,B)$. Then there exists a measurable function $\eta_{n,A,T_{0},B}\in\langle X\rangle^{m}L^{1}(\IR)$, a constant $c$ independent of $A$, $T_{0}$, and $B$, such that
    \begin{align}
        \int_{\IR}\langle x\rangle^{-m}\left|\eta_{n,A,T_{0},B}(x)\right|\Id x\leq c\left\|T_{0}\right\|_{S^{m,\frac{p}{p-n},\alpha}_{A}}\left\|B\right\|_{S^{m,p,\alpha}_{A}}^{n},
    \end{align}
    and such that for $f\in W^{m,n}_{0}$,
    \begin{align}
        \Tr\left(T_{0}\mathcal{J}\left(\dif_{n}f,A,\mathbf{B}\right)\right)=\int_{\IR}f^{(n)}(x)\eta_{n,A,T_{0},B}(x)\Id x.
    \end{align}
\end{thm}

\begin{proof}
    We consider Proposition \ref{prop:spectralshiftexist}. To show that the measure $\tau_{n,A,T_{0},\mathbf{B}}$ is absolutely continuous it suffices to show that the finite measure $\nu_{0,A,T_{0},\mathbf{B}}$ is absolutely continuous. Considering (\ref{eq:cor:multiopestimate:1}) we note that $\nu_{0,A,T_{0},\mathbf{B}}$ is given by the functional
    \begin{align}\label{eq:thm:spectralshiftfunctionexist:0}
        &\int_{\IR}\p_{x}^{n}\left((x+\Ii)^{m}f(x)\right)\Id\nu_{0,A,T_{0},\mathbf{B}}(x)\nonumber\\
        =&\Tr\left(T_{0}(A+\Ii)^{-m}\mathcal{J}\left(\dif_{n}\left((X+\Ii)^{m}f\right),A,\mathbf{B}\right)\right).
    \end{align}
    Let $(\psi_{l})_{l\in\IN}$ be an orthonormal basis of $H$, such that $\psi_{l}\in\dom A^{k}$ for all $k,l\in\IN$.
    Let $k,l\in\IN$, and define $P_{k}:=\one_{[-k,k]}(A)$ and $K_{l}:=\sum_{j=1}^{l}\langle\cdot,\psi_{j}\rangle\psi_{j}$.
    Let $B_{k,l}:=K_{l}P_{k}BP_{k}K_{l}$, $T_{0,k}:=P_{k}T_{0}P_{k}$, and denote $\mathbf{B}_{k,l}=(B_{k,l},\ldots,B_{k,l})$. Then $\mathbf{B}_{k,l}\in S^{0,\mathbf{1},0}_{A}$ and $T_{0}\in S^{0,\infty,0}_{A}$. Let $f\in C^{\infty}_{c}(\IR)$ and apply Proposition \ref{prop:divideddifference} to the right hand side of (\ref{eq:thm:spectralshiftfunctionexist:0}), with $T_{0,k}$ and $B_{k,l}$ replacing $T_{0}$ respectively $B$.
    \begin{align}\label{eq:thm:spectralshiftfunctionexist:1}
        &\Tr\left(T_{0,k}(A+\Ii)^{-m}\mathcal{J}\left(\dif_{n}\left((X+\Ii)^{m}f\right),A,\mathbf{B}_{k,l}\right)\right)\nonumber\\
        =&\Tr\left(T_{0,k}(A+\Ii)^{-m}T_{n}\left((X+\Ii)^{m}f,A,\mathbf{B}_{k,l}\right)\right)\nonumber\\
        =&\Tr\left(T_{0,k}(A+\Ii)^{-m}(R_{n-1}-T_{n-1})\left((X+\Ii)^{m}f,A,\mathbf{B}_{k,l}\right)\right)\nonumber\\
        =&(n-1)\int_{\IR}\p_{x}^{n-1}\left((x+\Ii)^{m}f(x)\right)\int_{0}^{1}(1-t)^{n-2}\nonumber\\&\left[\Id\tau_{n-1,A_{t},T_{0,k}(A+\Ii)^{-m},\mathbf{B}_{k,l}}(x)-\Id\tau_{n-1,A,T_{0,k}(A+\Ii)^{-m},\mathbf{B}_{k,l}}(x)\right]\nonumber\\
        =&\int_{\IR}\p_{x}^{n}\left((x+\Ii)^{m}f(x)\right)\kappa_{n,A,T_{0,k},\mathbf{B}_{k,l}}(x)\Id x,
    \end{align}
    where
    \begin{align}
        &\kappa_{n,A,T_{0,k},\mathbf{B}_{k,l}}(x):=(n-1)\int_{0}^{1}(1-t)^{n-2}\nonumber\\
        &\left[\tau_{n-1,A,T_{0,k}(A+\Ii)^{-m},\mathbf{B}_{k,l}}-\tau_{n-1,A_{t},T_{0,k}(A+\Ii)^{-m},\mathbf{B}_{k,l}}\right]((-\infty,x])\Id t,
    \end{align}
    which is well-defined because the measures $\tau_{n-1,A_{t},T_{0,k}(A+\Ii)^{-m},\mathbf{B}_{k,l}}$ are $t$-uniformly finite. On the other hand we have (\ref{eq:thm:spectralshiftfunctionexist:0}) which by variation on $f$ shows that $\nu_{0,A,T_{0},\mathbf{B}}$ is absolutely continuous. Denote its $L^{1}(\IR)$ density by
    \begin{align}
        \Id\nu_{0,A,T_{0,k},\mathbf{B}_{k,l}}(x)=:\eta_{n,A,T_{0,k},\mathbf{B}_{k,l}}(x)\Id x.
    \end{align}
    Now for $1\leq j\leq n$ let $\mathbf{B}^{j}_{k,l,l'}:=(B_{k,l},\ldots,B_{k,l},B_{k,l}-B_{k,l'},B_{k,l'},\ldots,B_{k,l'})$, where the difference sits in the $j$-th slot. For $l,l'\in\IN$ we obtain by multi-linearity of $\mathcal{J}$ and the estimates from Corollary \ref{cor:multiopestimate},
    \begin{align}\label{eq:thm:spectralshiftfunctionexist:2}
        &\int_{\IR}\left|\eta_{n,A,T_{0,k},\mathbf{B}_{k,l}}(x)-\eta_{n,A,T_{0,k},\mathbf{B}_{k,l'}}(x)\right|\Id x\nonumber\\
        =&\sup_{\stackrel{f\in C^{\infty}_{c}(\IR)}{\left\|\p_{x}^{n}\left((x+i)^{m}f(x)\right)\right\|_{\infty}=1}}\left|\Tr\left(T_{0,k}(A+\Ii)^{-m}\left[\mathcal{J}\left(\dif_{n}\left((X+\Ii)^{m}f\right),A,\mathbf{B}_{k,l}\right)\right.\right.\right.\nonumber\\
        &\left.\left.\left.-\mathcal{J}\left(\dif_{n}\left((X+\Ii)^{m}f\right),A,\mathbf{B}_{k,l'}\right)\right]\right)\right|\nonumber\\
        \leq&\sum_{j=1}^{n}\sup_{\stackrel{f\in C^{\infty}_{c}(\IR)}{\left\|\p_{x}^{n}\left((x+i)^{m}f(x)\right)\right\|_{\infty}=1}}\left|\Tr\left(T_{0,k}(A+\Ii)^{-m}\mathcal{J}\left(\dif_{n}\left((X+\Ii)^{m}f\right),A,\mathbf{B}_{k,l,l'}^{j}\right)\right)\right|\nonumber\\
        \leq& c'\left\|T_{0,k}\right\|_{S^{\frac{p}{p-n}}}\sup_{l''\in\IN}\left\|B_{k,l''}\right\|_{S^{p}}^{n}\left\|B_{k,l}-B_{k,l'}\right\|_{S^{p}}.
    \end{align}
    The last expression shows that $(\eta_{n,A,T_{0,k},\mathbf{B}_{k,l}})_{l\in\IN}$ is a Cauchy sequence in $L^{1}(\IR)$, because $K_{l}$ converges in strong operator topology to the identity, and also, by Lemma \ref{lem:schattconv}, $B_{k,l}$ converges in $S^{p}$ to $B_{k}$. Denote by $\eta_{n,A,T_{0,k},\mathbf{B}_{k}}$ its limit in $L^{1}(\IR)$. Let $\mathbf{B}_{k}:=(B_{k},\ldots,B_{k})$. Then (\ref{eq:thm:spectralshiftfunctionexist:2}) also shows convergence
    \begin{align}
        &\Tr\left(T_{0,k}(A+\Ii)^{-m}\mathcal{J}\left(\dif_{n}\left((X+\Ii)^{m}f\right),A,\mathbf{B}_{k,l}\right)\right)\nonumber\\
        \xrightarrow{l\to\infty}&\Tr\left(T_{0,k}(A+\Ii)^{-m}\mathcal{J}\left(\dif_{n}\left((X+\Ii)^{m}f\right),A,\mathbf{B}_{k}\right)\right),
    \end{align}
    and therefore, by variation in $f$ and because the measures and density are finite/integrable,
    \begin{align}
        \Id\nu_{0,A,T_{0,k},\mathbf{B}_{k}}(x)=\eta_{n,A,T_{0,k},\mathbf{B}_{k}}(x)\Id x.
    \end{align}
    Since $P_{k}$ commutes with $A$, and by commuting $P_{k}$ once under the trace we have
    \begin{align}\label{eq:thm:spectralshiftfunctionexist:4}
    &\Tr\left(T_{0,k}(A+\Ii)^{-m}\mathcal{J}\left(\dif_{n}\left((X+\Ii)^{m}f\right),A,\mathbf{B}_{k}\right)\right)\nonumber\\
    =&\Tr\left(T_{0,k}(A+\Ii)^{-m}\mathcal{J}\left(\dif_{n}\left((X+\Ii)^{m}f\right),A,\mathbf{B}\right)\right).
    \end{align}
    Additionally $P_{k}$ converges in the strong operator topology to the identity. Lemma \ref{lem:schattconv} implies that $T_{0,k}$ converges in $S^{m,\frac{p}{p-n},\alpha}_{A}$-norm to $T_{0}$. Thus for $k,k'\in\IN$ by (\ref{eq:thm:spectralshiftfunctionexist:4}) and Corollary \ref{cor:multiopestimate},
    \begin{align}\label{eq:thm:spectralshiftfunctionexist:3}
        &\int_{\IR}\left|\eta_{n,A,T_{0,k},\mathbf{B}_{k}}(x)-\eta_{n,A,T_{0,k'},\mathbf{B}_{k'}}(x)\right|\langle x\rangle^{-m}\Id x\nonumber\\
        =&\sup_{\stackrel{f\in C^{\infty}_{c}(\IR)}{\left\|\p_{x}^{n}\left((x+i)^{m}f(x)\right)\right\|_{\infty}=1}}\left|\Tr\left((T_{0,k}-T_{0,k'})(A+\Ii)^{-m}\mathcal{J}\left(\dif_{n}\left((X+\Ii)^{m}f\right),A,\mathbf{B}\right)\right)\right|\nonumber\\
        &\leq c'\left\|T_{0,k}-T_{0,k'}\right\|_{S^{m,\frac{p}{p-n},\alpha}_{A}}\left\|B\right\|_{S^{m,p,\alpha}_{A}}^{n}.
    \end{align}
    The last expression shows that $(\eta_{n,A,T_{0,k},\mathbf{B}_{k}})_{k\in\IN}$ is a Cauchy sequence in \break
    $\langle X\rangle^{m}L^{1}(\IR)$. Denote by $\kappa_{n,A,T_{0},\mathbf{B}}$ its limit in $\langle X\rangle^{m}L^{1}(\IR)$. On the other hand also
    \begin{align}
    &\Tr\left(T_{0,k}(A+\Ii)^{-m}\mathcal{J}\left(\dif_{n}\left((X+\Ii)^{m}f\right),A,\mathbf{B}_{k}\right)\right)\nonumber\\
    \xrightarrow{k\to\infty}&\Tr\left(T_{0}(A+\Ii)^{-m}\mathcal{J}\left(\dif_{n}\left((X+\Ii)^{m}f\right),A,\mathbf{B}\right)\right),
    \end{align}
    and thus
    \begin{align}
    &\Tr\left(T_{0}(A+\Ii)^{-m}\mathcal{J}\left(\dif_{n}\left((X+\Ii)^{m}f\right),A,\mathbf{B}\right)\right)\nonumber\\
    =&\int_{\IR}\p_{x}^{n}\left((x+\Ii)^{m}f(x)\right)\kappa_{n,A,T_{0},\mathbf{B}}(x)\Id x,
    \end{align}
    which by variation in $f$ and comparison with (\ref{eq:thm:spectralshiftfunctionexist:0}) implies that $\nu_{0,A,T_{0},\mathbf{B}}$ is absolutely continuous. The remainder of the claim then follows directly from the properties of $\tau_{n,A,T_{0},\mathbf{B}}$ in Proposition \ref{prop:spectralshiftexist}.
\end{proof}

To close this chapter, we show a well-known calculation rule for traces of multiple operator integrals in the context of Hypothesis \ref{hyp:mainhyp}.

\begin{lem}\label{lem:centraltorighttrace}
    Assume Hypothesis \ref{hyp:mainhyp} with $p=n$, and let $f\in W^{m,n}\cap W^{0,n}$. Let $\mathbf{B}:=(B,\ldots,B)$. Then
    \begin{align}\label{eq:lem:centraltorighttrace:2}
        \Tr\mathcal{J}(\dif_{n}f,A,\mathbf{B})=\frac{1}{n}\Tr\left(B\mathcal{J}(\dif_{n-1}f',A,\mathbf{B})\right).
    \end{align}
\end{lem}

\begin{proof}
We first remark that $f'\in W^{m,n-1}\cap W^{0,n-1}$ and therefore both operators inside the traces of (\ref{eq:lem:centraltorighttrace:2}) are indeed trace class by Corollary \ref{cor:multiopestimate}. For $k\in\IN$ let $P_{k}:=\one_{[-k,k]}(A)$, and define $B_{k}:=P_{k}AP_{k}$. Then $B_{k}$ converges to $B$ in $S^{m,n,\alpha}_{A}$ by Lemma \ref{lem:schattconv} because $P_{k}$ commutes with $A$. Moreover Hypothesis \ref{hyp:mainhyp} holds for $B_{k}$ instead of $B$ with uniform estimate
 \begin{align}
     \left\|B_{k}\right\|_{S^{m,n,\alpha}_{A}}\leq\left\|B\right\|_{S^{m,n,\alpha}_{A}}.
 \end{align}
 The operators $B_{k}$ are also in $S^{n}$ by Hypothesis \ref{hyp:mainhyp} and because $\langle A\rangle^{s}P_{k}$ is bounded for all $k\in\IN$, $s\in\IR$. Denote $\mathbf{B}_{k}:=(B_{k},\ldots,B_{k})$. Now let $\mathcal{F}\left(f^{(n)}\right)=\mu\in\operatorname{FM}(\IR)$. Then
    \begin{align}
        &\Tr\mathcal{J}(\dif_{n}f,A,\mathbf{B}_{k})=(2\pi)^{-1}\int_{\IR}\int_{s\in\Delta_{n}}\Tr\left(e^{\Ii s_{0}\xi A}\prod_{j=1}^{n}B_{k}e^{\Ii s_{j}\xi A}\right)\Id s~\Id\mu(\xi)\nonumber\\
        =&(2\pi)^{-1}\int_{\IR}\int_{s\in\Delta_{n}}\frac{1}{n}\sum_{l=0}^{n-1}\Tr\left(B_{k}\left(\prod_{j=1}^{l}e^{\Ii s_{n-j}\xi A}B_{k}\right)e^{\Ii (s_{0}+s_{n})\xi A}\right.\nonumber\\
        &\left.\left(\prod_{j=1}^{n-l-1}B_{k}e^{\Ii s_{j}\xi A}\right)\right)\Id s~\Id\mu(\xi)\nonumber\\
        =&\frac{1}{n}\sum_{l=0}^{n-1}\Tr\left(B_{k}\mathcal{J}(\lambda\mapsto\dif_{n}f(\lambda_{0},\ldots,\lambda_{l},\lambda_{l},\ldots,\lambda_{n-1}),A,\mathbf{B}_{k})\right)\nonumber\\
        =&\frac{1}{n}\sum_{l=0}^{n-1}\Tr\left(B_{k}\mathcal{J}(\p_{\lambda_{l}}\dif_{n-1}f,A,\mathbf{B}_{k})\right).
    \end{align}
    But
    \begin{align}
        &\sum_{l=0}^{n-1}\p_{\lambda_{j}}\dif_{n-1}f(\lambda)=\sum_{l=0}^{n-1}\int_{s\in\Delta_{n-1}}\p_{\lambda_{l}}f^{(n-1)}\left(\sum_{i=0}^{n-1}s_{i}\lambda_{i}\right)\Id s\nonumber\\
        =&\int_{s\in\Delta_{n-1}}f^{(n)}\left(\sum_{i=0}^{n-1}s_{i}\lambda_{i}\right)\Id s=\dif_{n-1}f'(\lambda).
    \end{align}
    Thus
    \begin{align}
        \Tr\mathcal{J}(\dif_{n}f,A,\mathbf{B}_{k})=\frac{1}{n}\Tr\left(B_{k}\mathcal{J}(\dif_{n-1}f',A,\mathbf{B}_{k})\right).
    \end{align}
    Since $P_{n}$ is a projection and commutes with $A$, and by commuting $P_{n}$ once under the traces we have
    \begin{align}\label{eq:lem:centraltorighttrace:1}
        \Tr\mathcal{J}(\dif_{n}f,A,(B_{k},B,\ldots,B))=\frac{1}{n}\Tr\left(B_{k}\mathcal{J}(\dif_{n-1}f',A,\mathbf{B})\right).
    \end{align}
    The left hand side of (\ref{eq:lem:centraltorighttrace:1}) converges to $\Tr\mathcal{J}(\dif_{n}f,A,B)$, and the right hand side to $\frac{1}{n}\Tr\left(B\mathcal{J}(\dif_{n-1}f',A,\mathbf{B})\right)$ by the estimate provided by Corollary \ref{cor:multiopestimate} and the linearity in $B_{k}$ of both sides.
\end{proof}

\section{Spectral shift functions adapted to Callias type operators}

In this chapter, we implement the constructed spectral shift function developed in Theorem \ref{thm:spectralshiftfunctionexist} in the context of Dirac-Schr\"odinger operators. Let us start by giving them and the setup we will assume.

\begin{defn}\label{def:calliasop}
    Let $H$ be a separable complex Hilbert space and let $A_{0}$ be self-adjoint in $H$.
    
    Let $d\in\IN$, $d\geq 3$ be odd. Let $(c_{j})_{j=1}^{d}$ be Clifford matrices in $\IC^{r}$ for some\footnote{the choice $r=2^{\frac{d-1}{2}}$ is minimal} $r\in\IN$, i.e.
    \begin{align}\label{eq:def:calliasop:1}
        c_{i}c_{j}+c_{j}c_{i}=-2\delta_{ij}\one_{\IC^{r}}
    \end{align}
    Let $c\nabla$ denote the self-adjoint extension (in $L^2(\IR^{d},\IC^{r}\otimes H)$) of the differential operator
    \begin{align}
        \sum_{j=1}^{d}c_{j}\p_{x^{j}}\otimes\one_{H}
    \end{align}
    to the domain $\dom c\nabla= W^{1,2}(\IR^{d},\IC^{r}\otimes H)$, the order $1$ Bochner-$L^2$-Sobolev space with values in $\IC^{r}\otimes H$.

    Denote by $\nabla_{R}:=\sum_{i=1}^{d}\frac{x^{i}}{|x|}\p_{x^{i}}$ the radial vector field, and let $c_{R}:=c\frac{x}{|x|}=\sum_{i=1}^{d}\frac{x^{i}}{|x|}c_{i}$. Let $c\nabla_{R}:=\sum_{i=1}^{d}\frac{x^{i}}{|x|}c_{i}\p_{x^{i}}$ be the radial part of the Dirac operator and $c\nabla_{\text{ang}}:=c\nabla-c\nabla_{R}$ its angular part.

    For $B$ a operator function on $\IR^{d}$ with domains $\dom\ B\left(x\right)=\dom\ A_{0}$, we write $B\in C^{k}_{b}\left(\IR^{d},A_{0},m\right)$, for $k\in\IN_{0}$, $m\geq 0$, if $\langle A_{0}\rangle^{\beta}B\left(x\right)\langle A_{0}\rangle^{-1-\beta}$, $x\in\IR^{d}$, is a family of closable operators with closures in $B\left(H\right)$ for any $\beta\in\left[-m,m\right]$, and $x\mapsto\overline{\langle A_{0}\rangle^{\beta}B\left(x\right)\langle A_{0}\rangle^{-1-\beta}}\in C^{k}_{b}\left(\IR^{d},B\left(H\right)\right)$, where $C^{k}_{b}$ is the the space of $k$-times continuously differentiable bounded functions in the strong operator topology with bounded derivatives, where boundedness is with respect to the $B\left(H\right)$ operator norm. As usual, $C^{k}_{b}\left(\IR^{d}\right)$ just denotes the scalar/matrix valued-$C^{k}_{b}$-functions. All involved spaces come equipped with the obvious norm given by summing the appropriate $L^{\infty}$-norms of the derivatives up to order $k$.

    For $B\in C^{1}_{b}(\IR^{d},A_{0},1)$ define
    \begin{align}
		\rho_{z}\left(B\right):=&\left\|B\langle A_{0}\rangle^{-1}_{z}\right\|_{L^{\infty}\left(\IR^{d},B\left(H\right)\right)}+\left\|\langle A_{0}\rangle B\langle A_{0}\rangle^{-1}\langle A_{0}\rangle^{-1}_{z}\right\|_{L^{\infty}\left(\IR^{d},B\left(H\right)\right)}\nonumber\\
		&+\left\|\nabla B\langle A_{0}\rangle^{-1}\langle M_{0}\rangle^{-1}_{z}\right\|_{L^{\infty}\left(\IR^{d},B\left(H\right)^{d}\right)},
	\end{align}
    and for $B\in C^{1}_{b}(\IR^{d},A_{0},|\beta|+1)$, $\beta\in\IR$,
    \begin{align}
        \rho_{z}^{\beta}\left(B\right):=\rho_{z}(\langle A_{0}\rangle^{\beta}B\langle A_{0}\rangle^{-\beta})
    \end{align}
    
    Denote by $A$ the operator function on $\IR^{d}$ given by $A(x):=A_{0}+B(x)$.
    
    On $\dom D_{B}:=\dom c\nabla\cap L^2\left(\IR^{d},\IC^{r}\otimes\dom A_{0}\right)$ define the Dirac-Schr\"odinger operator
    \begin{align}
        (D_{B}f)(x):=\Ii c\nabla f(x)+(\one_{\IC^{r}}\otimes A(x))f(x),\ x\in\IR^{d}.
    \end{align}
\end{defn}

The next hypothesis provides a list with all conditions assumed in \cite{Furst1}, which ensure that the principal trace formula we cite below holds.

\begin{hyp}\label{hyp:calliashyp}
	Let $\alpha\geq 1$, and $N\geq\lfloor\frac{\alpha-1}{2}\left(d+1\right)\rfloor+1$. Let $B\in C^{1}_{b}\left(\IR^{d},A_{0},2N+1\right)$. Assume that
	\begin{enumerate}
		\item \begin{align}\label{eq:hyp:calliashyp:1}
			\lim_{\mathrm{dist}\left(z,\left(-\infty,0\right]\right)\to\infty}\rho_{z}^{-2N}\left(B\right)+\rho_{z}^{2N}\left(B\right)=0,
		\end{align}
		and that
		\begin{align}\label{eq:hyp:calliashyp:2}
            \sup_{x\in\IR^{d}}\langle x\rangle\left\|\nabla B(x)\right\|_{S^{2N+1,d,\alpha}_{A_{0}}}&<\infty,\nonumber\\
			\exists\epsilon>0:~\sup_{x\in\IR^{d}}\langle x\rangle^{1+\epsilon}\left\|\p_{R}B(x)\right\|_{S^{2N+1,d,\alpha}_{A_{0}}}&<\infty,
		\end{align}
		where $\p_{R}:=\sum_{i=1}^{d}\frac{x^{i}}{\left|x\right|}\p_{i}$ denotes the radial vector field.
		\item Let $S_{1}\left(0\right)$ denote the $\left(d-1\right)$-dimensional unit sphere. For $\phi\in\dom\ A_{0}$, a.e. $y\in S_{1}\left(0\right)$, a.e. $x\in\IR^{d}$, and $\gamma\in\IN_{0}^{d}$ with $\left|\gamma\right|\leq 1$, assume,
		\begin{align}\label{eq:hyp:calliashyp:3}
			\lim_{R\to\infty}R^{\left|\gamma\right|}\left\|\left(\p^{\gamma}B\left(Ry+x\right)-\p^{\gamma}B\left(Ry\right)\right)\phi\right\|_{H}&=0. 
		\end{align}
	\end{enumerate}
\end{hyp}

\begin{rem}
	By a Kato-Rellich argument Hypothesis \ref{hyp:calliashyp} implies that for a.e. $x_{0}\in\IR^{d}$, $A\left(x_{0}\right)=A_{0}+B\left(x_{0}\right)$ can take the role of the model operator $A_{0}$.
\end{rem}

The main result of \cite{Furst1} is a principal trace formula for the Dirac-Schr\"odinger operator $D_{B}$.

\begin{thm}\label{thm:principaltrace}
    Assume Hypothesis \ref{hyp:calliashyp}. Then for $t>0$
    \begin{align}
        &\Tr_{L^2(\IR^{d},H)}\tr_{\IC^{r}}\left(e^{-tD_{B}^{\ast}D_{B}}-e^{-tD_{B}D_{B}^{\ast}}\right)\nonumber\\
        =&\frac{2}{d}(4\pi)^{-\frac{d}{2}}t^{\frac{d}{2}}\int_{\IR^{d}}\int_{s\in\Delta_{d-1}}\Tr_{\IC^{r}\otimes H}\left(\prod_{j=0}^{d-1}\left(\Ii c\nabla A_{\phi}\right)(x)e^{-ts_{j}A_{\phi}^{2}(x)}\right)\Id s~\Id x,
    \end{align}
    where $\phi\in C_{c}^{\infty}(\IR^{d})$, such that $\phi\equiv 1$ near $0$ is arbitrary and $A_{\phi}:=A_{0}+(1-\phi)(A-A_{0})$. In particular, the trace is independent of the choice of $\phi$.
\end{thm}

We will need the following elementary result for traces of Clifford matrices.

\begin{lem}\label{lem:clifford}
		\begin{enumerate}
			\item For $n\in\IN$ and $\gamma\in\left\{1,\ldots,d\right\}^{n}$,
			\begin{align}
				\tr_{\IC^{r}}\left(\prod_{j=1}^{n}c^{\gamma_{j}}\right)=0,
			\end{align}
			if ($n$, $d$ are odd, and $n<d$) or ($n$ is odd, and $d$ is even).
			\item For $d$ odd, and $\alpha\in\left\{1,\ldots,d\right\}^{d}$,
			\begin{align}\label{eq:lem:clifford:1}
				\tr_{\IC^{r}}\left(\prod_{j=1}^{d}c^{\gamma_{j}}\right)=\left(2\Ii\right)^{\frac{d-1}{2}}\left(-\Ii\right)^{d}\epsilon_{\gamma},
			\end{align}
			where $\epsilon_{\gamma}=\epsilon_{\gamma_{1}\ldots\gamma_{d}}$ is the Levi-Civita symbol.
		\end{enumerate}
	\end{lem}
	
	\begin{proof}
		\begin{enumerate}
			\item Due to the anti commutator relation (\ref{eq:def:calliasop:1}), it is sufficient to consider only those $\gamma$ for which all $\gamma_{j}$, $j\in\left\{1,\ldots,n\right\}$, are pairwise different. Since $n$ is odd, we have $n<d$ in any case. Thus there exists $m\in\left\{1,\ldots,d\right\}$ such that $m\neq\gamma_{j}$, $j\in\left\{1,\ldots,n\right\}$. Then
			\begin{align}
				\tr_{\IC^{r}}\left(\prod_{j=1}^{n}c^{\alpha_{j}}\right)&=-\tr_{\IC^{r}}\left(\left(c^{m}\right)^{2}\prod_{j=1}^{n}c^{\alpha_{j}}\right)\nonumber\\
				&=-\tr_{\IC^{r}}\left(c^{m}\prod_{j=1}^{n}c^{\alpha_{j}}c^{m}\right)=\left(-1\right)^{n+1}\tr_{\IC^{r}}\left(\left(c^{m}\right)^{2}\prod_{j=1}^{n}c^{\alpha_{j}}\right)\nonumber\\
				&=-\tr_{\IC^{r}}\left(\prod_{j=1}^{n}c^{\alpha_{j}}\right).
			\end{align}
			\item Statement (\ref{eq:lem:clifford:1}) follows up to the constant immediately from the anti commutator relation (\ref{eq:def:calliasop:1}), the definition of $\epsilon_{\gamma}$, and the first part of this proof. The constant in (\ref{eq:lem:clifford:1}) can be calculated inductively by the representation of Clifford matrices of odd dimension $d$ by those of dimension $d-2$. The sign convention corresponds to the choice $c^{j}=-\Ii\sigma^{j}$, where $\sigma^{j}$ are the classical Pauli matrices.
		\end{enumerate}
	\end{proof}

    The conditions on the operator family $A=A_{0}+B$ suffice to define the spectral shift function which allows us to re-express the right hand side of Theorem \ref{thm:principaltrace}, which is the first main result of this chapter.

\begin{prop}\label{prop:spectralcallias}
    Assume Hypothesis \ref{hyp:calliashyp}, and denote $\mathbf{\Ii c\nabla A_{\phi}(x)}$\break
    $:=\left(\Ii c\nabla A_{\phi}(x),\ldots,\Ii c\nabla A_{\phi}(x)\right)$. Then there exists a unique function $\eta_{d,A_{0},B}\in\langle X\rangle^{N}L^{1}(\IR_{\geq 0})$, such that for $f\in W^{N,d}_{0}$,
    \begin{align}\label{eq:prop:spectralcallias:1}
        &\int_{\IR^{d}}\Tr_{\IC^{r}\otimes H}\mathcal{J}\left(\dif_{d}f,A_{\phi}^{2}(x),\mathbf{\Ii c\nabla A_{\phi}(x)}\right)\Id x\nonumber\\
        =&\frac{1}{d}\int_{\IR^{d}}\Tr_{\IC^{r}\otimes H}\left(\Ii c\nabla A_{\phi}(x)\mathcal{J}\left(\dif_{d-1}f',A_{\phi}^{2}(x),\mathbf{\Ii c\nabla A_{\phi}(x)}\right)\right)\Id x\nonumber\\
        =&\int_{0}^{\infty}f^{(d)}(\lambda)\eta_{d,A_{0},B}(\lambda)\Id\lambda.
    \end{align}
\end{prop}

\begin{proof}
    Hypothesis \ref{hyp:calliashyp} implies that Hypothesis \ref{hyp:mainhyp} holds with $n=d$, $p=d$, $m=N$, $A=\one_{\IC^{r}}\otimes A_{\phi}^{2}(x)$, $B=\Ii c\nabla A_{\phi}(x)$ for each $x\in\IR^{d}$. Lemma \ref{lem:centraltorighttrace} implies
    \begin{align}
        \Tr_{\IC^{r}\otimes H}\mathcal{J}\left(\dif_{d}f,A,\mathbf{B}\right)=\frac{1}{d}\Tr_{\IC^{r}\otimes H}\left(B\mathcal{J}\left(\dif_{d-1}f',A,\mathbf{B}\right)\right).
    \end{align}
    Now extend $\p_{R}$ by $\p_{\psi_{1}},\ldots,\p_{\psi_{d-1}}$ to an orthonormal coordinate frame $\p_{y^{1}},\ldots,\p_{y^{d}}$ of $\IR^{d}\backslash\{0\}$ (polar coordinates). Denote $c_{y^{j}}:=\sum_{i=1}^{d}c_{i}\frac{\p y^{j}}{\p x^{i}}$. By Lemma \ref{lem:clifford} we have
    \begin{align}
        &\Tr_{\IC^{r}\otimes H}\left(B\mathcal{J}\left(\dif_{d-1}f',A,\mathbf{B}\right)\right)\nonumber\\
        =&\Ii^{d}\Tr_{\IC^{r}\otimes H}\left(c\nabla A_{\phi}(x)\mathcal{J}\left(\dif_{d-1}f',A_{\phi}^{2}(x),c\nabla A_{\phi}(x)\right)\right)\nonumber\\
        =&\Ii^{d}\tr_{\IC^{r}}\left(c_{y^{1}}\cdot\ldots\cdot c_{y^{d}}\right)(x)\sum_{\sigma\in\mathfrak{S}_{d}}(-1)^{|\sigma|}\nonumber\\
        &\Tr_{H}\left(
        \p_{y^{\sigma{1}}}A_{\phi}(x)\mathcal{J}\left(\dif_{d-1}f',A_{\phi}^{2}(x),\left(\p_{y^{\sigma(2)}}A_{\phi}(x),\ldots,\p_{y^{\sigma(d)}}A_{\phi}(x)\right)\right)\right).
    \end{align}
    By cyclically commuting under the trace $\Tr_{H}$ we may arrange the first derivative to be $\p_{R}$, and by using Lemma \ref{lem:clifford} again, we obtain
    \begin{align}
        &\Tr_{\IC^{r}\otimes H}\left(B\mathcal{J}\left(\dif_{d-1}f',A,\mathbf{B}\right)\right)\nonumber\\
        =&\Ii^{d}d\cdot\Tr_{\IC^{r}\otimes H}\left(c\nabla_{R}A_{\phi}(x)\mathcal{J}\left(\dif_{d-1}f',A_{\phi}^{2}(x),c\nabla_{\text{ang}}A_{\phi}(x)\right)\right),
    \end{align}
    and thus
    \begin{align}\label{eq:prop:spectralcallias:3}
        \Tr_{\IC^{r}\otimes H}\mathcal{J}\left(\dif_{d}f,A,\mathbf{B}\right)=\Tr_{\IC^{r}\otimes H}\left(T_{0}\mathcal{J}\left(\dif_{d-1}f',A,\mathbf{B}_{\text{ang}}\right)\right),
    \end{align}
    where $T_{0}:=T_{0}(x):=\Ii c\nabla_{R}A_{\phi}(x)$, and $\mathbf{B}_{\text{ang}}:=\mathbf{B}_{\text{ang}}(x)=(B_{\text{ang}},\ldots,B_{\text{ang}})(x)$ with $B_{\text{ang}}(x):=\Ii c\nabla_{\text{ang}}A_{\phi}(x)$. We apply Theorem \ref{thm:spectralshiftfunctionexist} to the right hand side of (\ref{eq:prop:spectralcallias:3}) and obtain a function $\eta_{d-1,A,T_{0},\mathbf{B}_{\text{ang}}}^{x}\in\langle X\rangle^{N}L^{1}(\IR)$ with
    \begin{align}\label{eq:prop:spectralcallias:4}
        &\int_{\IR}\langle \lambda\rangle^{-N}\left|\eta_{d-1,A,T_{0},\mathbf{B}_{\text{ang}}}^{x}(\lambda)\right|\Id\lambda\leq c\left\|\nabla_{R}A_{\phi}(x)\right\|_{S^{2N+1,d,\alpha}_{A_{0}}}\left\|\nabla A_{\phi}(x)\right\|_{S^{2N+1,d,\alpha}_{A_{0}}}^{d-1}\nonumber\\
        \leq &c'\langle x\rangle^{-d-\epsilon},
    \end{align}
    such that
    \begin{align}
        &\Tr_{\IC^{r}\otimes H}\mathcal{J}\left(\dif_{d}f,A,\mathbf{B}\right)=\Tr_{\IC^{r}\otimes H}\left(T_{0}\mathcal{J}\left(\dif_{d-1}f',A,\mathbf{B}_{\text{ang}}\right)\right)\nonumber\\
        =&\int_{\IR}f^{(d)}(\lambda)\eta_{d-1,A,T_{0},\mathbf{B}_{\text{ang}}}^{x}(\lambda)\Id\lambda.
    \end{align}
    Lemma \ref{lem:measureable} implies that $x\mapsto\left(\lambda\mapsto\eta_{d-1,A,T_{0},\mathbf{B}_{\text{ang}}}^{x}\Id\lambda\right)$ is strongly measurable as a map from $\IR^{d}$ to $\langle X\rangle^{N}\operatorname{FM}(\IR)$, and thus $x\mapsto\eta_{d-1,A,T_{0},\mathbf{B}_{\text{ang}}}^{x}$ is strongly measurable into $\langle X\rangle^{N}L^{1}(\IR)$. By (\ref{eq:prop:spectralcallias:4}) it follows that $\IR^{d}\ni x\mapsto\eta_{d-1,A,T_{0},\mathbf{B}_{\text{ang}}}^{x}\in\langle X\rangle^{N}L^{1}(\IR)$ is Bochner integrable, and we may define
    \begin{align}
        \eta_{d,A_{0},B}:=\int_{\IR^{d}}\eta_{d-1,A,T_{0},\mathbf{B}_{\text{ang}}}^{x}\Id x,
    \end{align}
    with
    \begin{align}
        \int_{\IR}\left|\eta_{d,A_{0},B}(\lambda)\right|\langle\lambda\rangle^{-N}\Id\lambda\leq c'\int_{\IR^{d}}\left\|\nabla_{R}A_{\phi}(x)\right\|_{S^{2N+1,d,\alpha}_{A_{0}}}\left\|\nabla A_{\phi}(x)\right\|_{S^{2N+1,d,\alpha}_{A_{0}}}^{d-1}\Id x.
    \end{align}
    Fubini's theorem finally implies (\ref{eq:prop:spectralcallias:1}). The uniqueness statement and the support restriction to $\IR_{\geq 0}$ follow from Lemma \ref{lem:support}.
\end{proof}

Although the stated conditions in Hypothesis \ref{hyp:calliashyp} are sufficient to obtain the principal trace formula in Theorem \ref{thm:principaltrace}, we will have to strengthen the conditions posed on the operator family $B$ further, such that we may define an appropriate spectral shift function associated to $D_{B}$ and $D_{B}^{\ast}$ within the framework developed in the previous chapter of this paper.

\begin{hyp}\label{hyp:smoothhyp}
    Assume Hypothesis \ref{hyp:calliashyp}, and without loss assume that $B\equiv 0$ near $0$. Additionally, assume that
    \begin{align}
        \forall\epsilon>0:\forall\gamma\in\IN_{0}^{d},|\gamma|\leq 2N:&~\sup_{x\in\IR^{d}}\langle x\rangle\left\|\p^{\gamma}\nabla B\right\|_{S^{0,d,\alpha+|\gamma|}_{A_{0}}}<\infty,\\
        \exists\delta>0:\forall\gamma\in\IN_{0}^{d},|\gamma|\leq 2N:&~\sup_{x\in\IR^{d}}\langle x\rangle^{1+\delta}\left\|\p_{R}B\right\|_{S^{2N+1,d-\delta,\alpha}_{A_{0}}}\\
        &+\sup_{x\in\IR^{d}}\langle x\rangle^{1+\delta}\left\|\p^{\gamma}\p_{R} B\right\|_{S^{0,d-\delta,\alpha+|\gamma|}_{A_{0}}}<\infty.
    \end{align}
\end{hyp}

\begin{lem}\label{lem:bulkfamily}
    If the operator family $A$ satisfies Hypothesis \ref{hyp:smoothhyp}, then also the operator family $A_{\phi}:=A_{0}+(1-\phi)B$ satisfies Hypothesis \ref{hyp:smoothhyp}, for $\phi\in C_{c}^{\infty}(\IR^{d})$, $\phi\equiv 1$ near $0$.
\end{lem}

\begin{proof}
    We note
    \begin{align}
        \nabla A_{\phi}=(1-\phi)\nabla A-(\nabla\phi)A,
    \end{align}
    and that integration on line segments shows that $B$ is a continuous function with values in $S^{2N+1,d,\alpha}_{A_{0}}$. Because $\nabla\phi$ is compactly supported and thus exhibits arbitrary decay, it follows that conditions (\ref{eq:hyp:calliashyp:1}),(\ref{eq:hyp:calliashyp:2}) and (\ref{eq:hyp:calliashyp:3}) of Hypothesis \ref{hyp:calliashyp} are also satisfied for $A_{\phi}$. A similar result is obtained for higher order derivatives and thus yields Hypothesis \ref{hyp:smoothhyp} for $A_{\phi}$.
\end{proof}

With regard of the independence stated in Theorem \ref{thm:principaltrace} and Lemma \ref{lem:bulkfamily}, from now on we may assume without loss of generality that $B\equiv 0$ in a neighbourhood of $0$. The benefit of this restriction is that we will have no trouble defining, for example, $c\nabla_{R}A=c\nabla_{R}B$ at $0$.

\begin{defn}
     For an operator-valued function $\IR^{d}\ni x\mapsto T(x)$ with constant domains $\dom T(x)=\dom T_{0}$, denote by $M_{T}$ the multiplication operator in $L^2(\IR^{d},\IC^{r}\otimes H)$ given by
    \begin{align}
        (M_{T}f)(x):=T(x)f(x),\ f\in L^2(\IR^{d},\dom T_{0})=:\dom M_{T}.
    \end{align}
\end{defn}

\begin{lem}\label{lem:hdomproperty}
    Assume that for $n\in\IN$, $n\leq 2N$
    \begin{align}
        \sup_{x\in\IR^{d}}\left\|\langle A_{0}\rangle^{k}B(x)\langle A_{0}\rangle^{-k}\langle A_{0}\rangle_{z}^{-1}\right\|_{B(H)}\xrightarrow{z\to\infty}0,\ k\in\IN_{0},\ k\leq n-1,\nonumber\\
        \sup_{x\in\IR^{d}}\left\|\p^{\beta}B(x)\langle A_{0}\rangle^{-|\beta|}\langle A_{0}\rangle_{z}^{-1}\right\|_{B(H)}\xrightarrow{z\to\infty}0,\ \beta\in\IN_{0}^{d},\ |\beta|\leq n-1.
    \end{align}
    On $\dom D_{0}=W^{1,2}(\IR^{d},\IC^{r}\otimes H)\cap L^2(\IR^{d},\IC^{r}\otimes\dom A_{0})$ define the operators
     \begin{align}
         T_{B,0}:=D_{B},\ T_{B,1}:=D_{B}^{\ast}.
     \end{align}
     Then for $\eta\in\{0,1\}^{n}$ the operator
     \begin{align}
         S_{B,\eta}:=\prod_{j=1}^{n}T_{B,\eta_{j}},
     \end{align}
     is closed with
     \begin{align}
         \dom S_{B,\eta}:=\dom S_{0,\eta}=\dom H_{0}^{\frac{n}{2}}=W^{n,2}(\IR^{d},\IC^{r}\otimes H)\cap L^2(\IR^{d},\IC^{r}\otimes\dom A_{0}^{n}).
     \end{align}
     Moreover if $S_{B,\eta}$ is symmetric, then it is self-adjoint.
\end{lem}

\begin{proof}
    We first observe that $A_{0}$ and $c\nabla$ commute, so we may diagonalize $A_{0}$ and $c\nabla$ simultaneously. The equivalence of the graph norm of $H_{0}^{\frac{n}{2}}$ and the norm of $W^{n,2}(\IR^{d},\IC^{r}\otimes H)\cap L^2(\IR^{d},\IC^{r}\otimes\dom A_{0}^{n})$ then follows from the fact that the polynomials $(x^2+y^2+1)^{n}$ and $x^{2n}+y^{2n}+2$ are uniformly bounded with respect to each other (with a constant only dependent on $n$). The same type of argument implies that $S_{0,\eta}$ is closed on $W^{n,2}(\IR^{d},\IC^{r}\otimes H)\cap L^2(\IR^{d},\IC^{r}\otimes\dom A_{0}^{n})$. To show that $S_{B,\eta}$ is closed on the same domain (and self-adjoint if symmetric), we want to apply the Kato-Rellich theorems \cite{Kato}[Theorem 4.1.1 and Theorem 5.4.3]. It therefore suffices to show that
    \begin{align}
        \lim_{z\to\infty}\left\|(S_{B,\eta}-S_{0,\eta})(H_{0}+z)^{-\frac{n}{2}}\right\|_{B\left(L^2(\IR^{d},\IC^{r}\otimes H)\right)}=0
    \end{align}
    We proceed by induction on $n\in\IN_{0}$. For $n=0$ the statement is trivially true. For $n\in\IN$ we write
    \begin{align}
        S_{B,\eta}-S_{0,\eta}=\sum_{j=1}^{n-1}\left(\prod_{l=1}^{j-1}T_{B,\eta_{l}}\right)M_{B}\left(\prod_{l=j+1}^{n}T_{0,\eta_{l}}\right)
    \end{align}
    By induction, $\prod_{l=1}^{j-1}T_{B,\eta_{l}}$ is relatively bounded with respect to $H_{0}^{\frac{j-1}{2}}$, and $\prod_{l=j+1}^{n}T_{0,\eta_{l}}$ commutes with $H_{0}$ and is relatively bounded by $H_{0}^{\frac{n-j}{2}}$, so it suffices to show that
    \begin{align}\label{eq:lem:hdomproperty:1}
        \left\|M_{B}(H_{0}+z)^{-\frac{j}{2}}\right\|_{B\left(L^2\left(\IR^{d},\IC^{r}\otimes H\right),W^{j-1,2}(\IR^{d},\IC^{r}\otimes H)\cap L^2(\IR^{d},\IC^{r}\otimes\dom A_{0}^{j-1})\right)}\xrightarrow{z\to\infty}0.
    \end{align}
    We apply the product rule to multiplication with $B$, and the fact that the norm of a multiplication operator is given by the supremum of the operator norms in each fibre, and see that (\ref{eq:lem:hdomproperty:1}) holds if
    \begin{align}
        \sup_{x\in\IR^{d}}\left\|\langle A_{0}\rangle^{j-1}B(x)\langle A_{0}\rangle^{1-j}\langle A_{0}\rangle_{z}^{-1}\right\|_{B(H)}\xrightarrow{z\to\infty}0,\nonumber\\
        \sup_{x\in\IR^{d}}\left\|\p^{\beta}B(x)\langle A_{0}\rangle^{-|\beta|}\langle A_{0}\rangle_{z}^{-1}\right\|_{B(H)}\xrightarrow{z\to\infty}0,\ \beta\in\IN_{0}^{d},\ |\beta|\leq j-1.
    \end{align}
\end{proof}

A simple application of the previous lemma is the following result on convex combinations of the operators $D_{B}^{\ast}D_{B}$ and $D_{B}D_{B}^{\ast}$.

\begin{cor}\label{cor:uninormequiv}
     Assume Hypothesis \ref{hyp:smoothhyp}, and let $s\in[-\frac{1}{2},\frac{1}{2}]$. Then $H_{B,s}:=(\frac{1}{2}-s)D_{B}^{\ast}D_{B}+(\frac{1}{2}+s)D_{B}D_{B}^{\ast}$ is non-negative self-adjoint on $\dom H_{0}$, and the graph norms of $H_{0}^{k}$, and $H_{B,s}^{k}$ are equivalent, uniformly in $s$ for $1\leq k\leq 2N$.
\end{cor}

\begin{proof}
    For $\eta\in\{0,1\}^{2k}$ there are polynomials $p_{\eta}$, independent of $B$, such that
    \begin{align}
        H_{B,s}^{k}=\sum_{\eta\in\{0,1\}^{2k}}p_{\eta}(s)S_{B,\eta},
    \end{align}
    where the operators $S_{B,\eta}$ are defined in Lemma \ref{lem:hdomproperty}, and have equivalent graph-norms to $H_{0}$. It follows that the graph-norm of $H_{B,s}$ is $s$-uniformly equivalent to the graph norm of $H_{0}^{k}$ after passing to the supremum over $s\in[-\frac{1}{2},\frac{1}{2}]$ in the above decomposition of $H_{B,s}^{k}$. Because $H_{B,s}$ is symmetric it is self-adjoint on $\dom H_{0}$ by Lemma \ref{lem:hdomproperty}. Since $D_{B}^{\ast}D_{B}$, and $D_{B}D_{B}^{\ast}$ are  non-negative on $\dom H_{0}$, also $H_{B,s}$ is non-negative on $\dom H_{0}$.
\end{proof}

We are equipped to translate the properties of the operator family $B$ from Hypothesis \ref{hyp:smoothhyp} into conditions needed for the framework developed abstractly in the previous chapter for the multiplication operators associated with $\Ii c\nabla_{(R)}A$.

\begin{lem}\label{lem:hbproperty}
   Assume Hypothesis \ref{hyp:smoothhyp}. Then
    \begin{align}
        \forall\epsilon>0:~M_{\Ii c\nabla A}\in S^{N,d+\epsilon,\frac{\alpha}{2}+1}_{H_{B,s}},\nonumber\\
        \exists\delta>0:~ M_{\Ii c\nabla_{R}A}\in S^{N,d-\delta,\frac{\alpha}{2}+1}_{H_{B,s}},
    \end{align}
    and
    \begin{align}
        \forall\epsilon>0:~\sup_{s\in[-\frac{1}{2},\frac{1}{2}]}\left\|M_{\Ii c\nabla A}\right\|_{S^{N,d+\epsilon,\frac{\alpha}{2}+1}_{H_{B,s}}}<\infty,\nonumber\\
        \exists\delta>0:~ \sup_{s\in[-\frac{1}{2},\frac{1}{2}]}\left\|M_{\Ii c\nabla_{R}A}\right\|_{S^{N,d-\delta,\frac{\alpha}{2}+1}_{H_{B,s}}}<\infty.
    \end{align}
\end{lem}

\begin{proof}
    Because the graph norms of $H_{B,s}^{\frac{k}{2}}$ and $H_{0}^{\frac{k}{2}}$ are $s$-uniformly equivalent by Corollary \ref{cor:uninormequiv}, it suffices to show
    \begin{align}
        \forall\epsilon>0:~M_{\Ii c\nabla A}\in S^{N,d+\epsilon,\frac{\alpha}{2}+1}_{H_{0}},\nonumber\\
        \exists\delta>0:~ M_{\Ii c\nabla_{R}A}\in S^{N,d-\delta,\frac{\alpha}{2}+1}_{H_{0}}.
    \end{align}
    These conditions are satisfied if the operators
    \begin{align}
        T_{1}:=\langle H_{0}\rangle^{\frac{\alpha}{2}+N+1},~S_{1}:=\overline{\langle H_{0}\rangle^{N}M_{\Ii c\nabla A}},\nonumber\\
        T_{2}:=\langle H_{0}\rangle^{\frac{\alpha}{2}+N+1},~S_{2}:=\overline{\langle H_{0}\rangle^{N}M_{\Ii c\nabla_{R}A}},
    \end{align}
    satisfy for all $z\in[0,1]$
    \begin{align}
        \overline{T_{i}^{-z}S_{i}T_{i}^{-1+z}}\in S^{p_{i}},
    \end{align}
    where $p_{1}:=d+\epsilon$, and $p_{2}:=d-\delta$. We apply \cite{GLST}[Theorem 3.2], which is a complex interpolation theorem, so it is enough to show that
    \begin{align}
        S_{i}T_{i}^{-1},~S_{i}^{\ast}T_{i}^{-1}\in S^{p_{i}},
    \end{align}
    which means
    \begin{align}\label{eq:lem:hbproperty:1}
        \forall\epsilon>0:~\langle H_{0}\rangle^{N}M_{\Ii c\nabla A}\langle H_{0}\rangle^{-\frac{\alpha}{2}-N-1},~\Ii c\nabla A\langle H_{0}\rangle^{-\frac{\alpha}{2}-1}\in S^{d+\epsilon},\nonumber\\
        \exists\delta>0:~ \langle H_{0}\rangle^{N}M_{\Ii c\nabla_{R} A}\langle H_{0}\rangle^{-\frac{\alpha}{2}-N-1},~\Ii c\nabla_{R} A\langle H_{0}\rangle^{-\frac{\alpha}{2}-1}\in S^{d-\delta}.
    \end{align}
    Since the graph norm of $H_{0}^{N}$ is equivalent to the norm of $W^{2N,2}(\IR^{d},\IC^{r}\otimes H)\cap L^2(\IR^{d},\IC^{r}\otimes\dom A_{0}^{2N})$, (\ref{eq:lem:hbproperty:1}) is satisfied if
    \begin{align}\label{eq:lem:hbproperty:2}
        \forall\epsilon>0:\forall\gamma\in\IN_{0}^{d},|\gamma|\leq 2N:&~M_{A_{0}}^{2N}\Ii c\nabla A\langle H_{0}\rangle^{-\frac{\alpha}{2}-N-1},\nonumber\\
        &\p^{\gamma}\Ii c\nabla A\langle H_{0}\rangle^{-\frac{\alpha}{2}-N-1},~\Ii c\nabla A\langle H_{0}\rangle^{-\frac{\alpha}{2}-1}\in S^{d+\epsilon},\nonumber\\
        \exists\delta>0:\forall\gamma\in\IN_{0}^{d},|\gamma|\leq 2N:&~ M_{A_{0}}^{2N}\Ii c\nabla_{R} A\langle H_{0}\rangle^{-\frac{\alpha}{2}-N-1},\nonumber\\
        &\p^{\gamma}\Ii c\nabla_{R} A\langle H_{0}\rangle^{-\frac{\alpha}{2}-N-1},~\Ii c\nabla_{R} A\langle H_{0}\rangle^{-\frac{\alpha}{2}-1}\in S^{d-\delta}.
    \end{align}
    We apply the Leibniz rule, and recall that $H_{0}^{\frac{k}{2}}$ commutes with $A_{0}$ and $\di$ and relatively bounds $A_{0}^{k}$ and $\di^{k}$, and thus (\ref{eq:lem:hbproperty:2}) holds if
    \begin{align}\label{eq:lem:hbproperty:3}
        &\forall\epsilon>0:\forall\gamma\in\IN_{0}^{d},|\gamma|\leq 2N:~M_{A_{0}}^{2N}\Ii c\nabla A\langle M_{A_{0}}\rangle^{-2N-\alpha}\langle\di\rangle^{-2},\nonumber\\
        &\p^{\gamma}\left(\Ii c\nabla A\right)\langle M_{A_{0}}\rangle^{-
        |\gamma|-\alpha}\langle\di\rangle^{-2},~\Ii c\nabla A\langle M_{A_{0}}\rangle^{-\alpha}\langle\di\rangle^{-2}\in S^{d+\epsilon},\nonumber\\
        &\exists\delta>0:\forall\gamma\in\IN_{0}^{d},|\gamma|\leq 2N:~M_{A_{0}}^{2N}\Ii c\nabla_{R} A\langle M_{A_{0}}\rangle^{-2N-\alpha}\langle\di\rangle^{-2},\nonumber\\
        &\p^{\gamma}\left(\Ii c\nabla_{R} A\right)\langle M_{A_{0}}\rangle^{-
        |\gamma|-\alpha}\langle\di\rangle^{-2},~\Ii c\nabla_{R}A\langle M_{A_{0}}\rangle^{-\alpha}\langle\di\rangle^{-2}\in S^{d-\delta}.
    \end{align}
    We apply \cite{Furst1}[Lemma 2.8], also noting that $\langle X\rangle^{-2}\in L^{d-\delta}(\IR^{d})$ for $\delta>0$ small enough, and see that (\ref{eq:lem:hbproperty:3}) is satisfied if Hypothesis \ref{hyp:smoothhyp} holds.
\end{proof}

The following proposition gives the construction of the spectral shift function which allows us to re-express the left hand side of the principal trace formula in Theorem \ref{thm:principaltrace}, and is the second main result of this chapter. We note that the order of the spectral shift function is $d+1$, which we have to rectify later.

\begin{prop}\label{prop:leftspectralshift}
    Assume Hypothesis \ref{hyp:smoothhyp}. Then there exists a unique function $\Xi_{d,A_{0},B}\in\langle X\rangle^{N+1}L^{1}(\IR_{\geq 0})$, such that for 
    $f\in W^{N,d+1}_{0}\cap\bigcap_{j=0}^{d+1}W^{0,j}_{0}$,
    \begin{align}\label{eq:prop:leftspectralshift:5}
         \Tr_{L^2(\IR^{d},H)}\tr_{\IC^{r}}\left(f\left(D_{B}^{\ast}D_{B}\right)-f\left(D_{B}D_{B}^{\ast}\right)\right)=\int_{0}^{\infty}f^{(d+1)}(\lambda)\Xi_{d,A_{0},B}(\lambda)\Id\lambda.
    \end{align}
\end{prop}

\begin{proof}
    The operator $H_{B,0}$ is non-negative self-adjoint on $\dom H_{0}$ by Corollary \ref{cor:uninormequiv}. Moreover on $\dom H_{0}$ we have
    \begin{align}
        D_{B}D_{B}^{\ast}=H_{B,0}+M_{\Ii c\nabla A},\ D_{B}^{\ast}D_{B}=H_{B,0}-M_{\Ii c\nabla A}.
    \end{align}
    By Proposition \ref{prop:divideddifference} we derive on $\dom H_{0}^{\frac{d}{2}}$ for $f\in\bigcap_{j=0}^{d}W^{0,j}$,
    \begin{align}
        &f\left(D_{B}^{\ast}D_{B}\right)-f\left(D_{B}D_{B}^{\ast}\right)\nonumber\\
        =&R_{d+1}\left(f,H_{B,0},-M_{\Ii c\nabla A}\right)+\sum_{k=0}^{d}\mathcal{J}\left(\dif_{k}f,H_{B,0},-M_{\Ii c\nabla A}\right)\nonumber\\
        &-R_{d+1}\left(f,H_{B,0},+M_{\Ii c\nabla A}\right)-\sum_{k=0}^{d}\mathcal{J}\left(\dif_{k}f,H_{B,0},+M_{\Ii c\nabla A}\right)\nonumber\\
        =&R_{d+1}\left(f,H_{B,0},-M_{\Ii c\nabla A}\right)-R_{d+1}\left(f,H_{B,0},+M_{\Ii c\nabla A}\right)\\
        &-2\sum_{k=0}^{\frac{d-1}{2}}\mathcal{J}\left(\dif_{2k+1}f,H_{B,0},+M_{\Ii c\nabla A}\right).
    \end{align}
    Clifford matrices commute with $H_{B,0}$ and thus Lemma \ref{lem:clifford} implies that
    \begin{align}
        \sum_{k=0}^{\frac{d-1}{2}}\tr_{\IC^{r}}\mathcal{J}\left(\dif_{2k+1}f,H_{B,0},+M_{\Ii c\nabla A}\right)=\tr_{\IC^{r}}\mathcal{J}\left(\dif_{d}f,H_{B,0},+M_{\Ii c\nabla A}\right).
    \end{align}
    We split $c\nabla A=c\nabla_{R}A+c\nabla_{\text{ang}}A$, and denote
    \begin{align}
        \mathbf{M}_{k}:=\left(M_{\Ii c\nabla_{\text{ang}} A},\ldots,M_{\Ii c\nabla_{\text{ang}} A},M_{\Ii c\nabla_{R} A},M_{\Ii c\nabla_{\text{ang}} A},\ldots,M_{\Ii c\nabla_{\text{ang}} A}\right),
    \end{align}
    with $M_{\Ii c\nabla_{R} A}$ in the $k$-th slot. By Lemma \ref{lem:clifford} we have a non-trivial trace only if exactly one factor $c\nabla_{R}A$ appears, and thus
    \begin{align}
        \tr_{\IC^{r}}\mathcal{J}\left(\dif_{d}f,H_{B,0},+M_{\Ii c\nabla A}\right)=\sum_{k=1}^{d}\tr_{\IC^{r}}\mathcal{J}\left(\dif_{d}f,H_{B,0},\mathbf{M}_{k}\right).
    \end{align}
    We obtain
    \begin{align}\label{eq:prop:leftspectralshift:1}
        &\tr_{\IC^{r}}\left(f\left(D_{B}^{\ast}D_{B}\right)-f\left(D_{B}D_{B}^{\ast}\right)\right)\nonumber\\
        =&\tr_{\IC^{r}}\left(R_{d+1}\left(f,H_{B,0},-M_{\Ii c\nabla A}\right)-R_{d+1}\left(f,H_{B,0},+M_{\Ii c\nabla A}\right)\right)\nonumber\\
        &-2\sum_{k=1}^{d}\tr_{\IC^{r}}\mathcal{J}\left(\dif_{d}f,H_{B,0},\mathbf{M}_{k}\right),
    \end{align}
    where all arguments of the trace $\tr_{\IC^{r}}$ on the right hand side are trace-class in $L^2\left(\IR^{d},\IC^{r}\otimes H\right)$ by Lemma \ref{lem:hbproperty}, and Proposition \ref{prop:divideddifference} for $f\in W^{m,d+1}\cap\bigcap_{j=0}^{d+1}W^{0,j}$. Let $(\psi_{l})_{l\in\IN}$ be an orthonormal basis of $L^2(\IR^{d},H)$, such that $v\otimes\psi_{l}\in\dom H_{B,0}^{k}$ for any $k,l\in\IN$ and $v\in\IC^{r}$. For $i,j\in\IN$ denote $P_{j}:=\one_{[-j,j]}(H_{B,0})$, $\widehat{K}_{i}:=\sum_{l=1}^{i}\langle\cdot,\psi_{l}\rangle \psi_{l}$, and $K_{i}:=\one_{\IC^{r}}\otimes\widehat{K}_{i}$. Let $\mathbf{M}_{k}^{i,j}$ be the operator vector obtained from applying $K_{i}P_{j}\cdot P_{j}K_{i}$ to each entry of $\mathbf{M}_{k}$. Then strong operator convergence of $K_{i}$, $P_{j}$ to $\one_{L^{2}(\IR^{d},\IC^{r}\otimes H)}$ as $i,j\to\infty$, and since $P_{j}$ commutes with $H_{B,0}$, the convergence improving property in Lemma \ref{lem:schattconv} show that the following iterated limit statement is valid.
    \begin{align}\label{eq:prop:leftspectralshift:2}
        &\sum_{k=1}^{d}\Tr_{L^2(\IR^{d},\IC^{r}\otimes H)}\mathcal{J}\left(\dif_{d}f,H_{B,0},\mathbf{M}_{k}\right)\nonumber\\
        =&\lim_{j\to\infty}\lim_{i\to\infty}\sum_{k=1}^{d}\Tr_{L^2(\IR^{d},\IC^{r}\otimes H)}\mathcal{J}\left(\dif_{d}f,H_{B,0},\mathbf{M}_{k}^{i,j}\right).
    \end{align}
    We apply Lemma \ref{lem:clifford} again, noting that $K_{i}$ and $P_{j}$ commute with any matrix in $\IC^{r\times r}$, and obtain
    \begin{align}
        &\sum_{k=1}^{d}\Tr_{L^2(\IR^{d},\IC^{r}\otimes H)}\mathcal{J}\left(\dif_{d}f,H_{B,0},\mathbf{M}_{k}^{i,j}\right)\nonumber\\
        =&\Tr_{L^2(\IR^{d},\IC^{r}\otimes H)}\mathcal{J}\left(\dif_{d}f,H_{B,0},K_{i}P_{j}M_{\Ii c\nabla A}P_{j}K_{i}\right)
    \end{align}
    We may apply Lemma \ref{lem:centraltorighttrace} and thus have
    \begin{align}
        &\Tr_{L^2(\IR^{d},\IC^{r}\otimes H)}\mathcal{J}\left(\dif_{d}f,H_{B,0},K_{i}P_{j}M_{\Ii c\nabla A}P_{j}K_{i}\right)\nonumber\\
        =&\frac{1}{d}\Tr_{L^2(\IR^{d},\IC^{r}\otimes H)}\left(K_{i}P_{j}M_{\Ii c\nabla A}P_{j}K_{i}\mathcal{J}\left(\dif_{d-1}f',H_{B,0},K_{i}P_{j}M_{\Ii c\nabla A}P_{j}K_{i}\right)\right).
    \end{align}
    We apply Lemma \ref{lem:clifford} after the split $c\nabla A=c\nabla_{R}A+c\nabla_{\text{ang}}A$, and commute under the trace to get
    \begin{align}
        &\frac{1}{d}\Tr_{L^2(\IR^{d},\IC^{r}\otimes H)}\left(K_{i}P_{j}M_{\Ii c\nabla A}P_{j}K_{i}\mathcal{J}\left(\dif_{d-1}f',H_{B,0},K_{i}P_{j}M_{\Ii c\nabla A}P_{j}K_{i}\right)\right)\nonumber\\
        =&\Tr_{L^2(\IR^{d},\IC^{r}\otimes H)}\left(K_{i}P_{j}M_{\Ii c\nabla_{R} A}P_{j}K_{i}\mathcal{J}\left(\dif_{d-1}f',H_{B,0},K_{i}P_{j}M_{\Ii c\nabla_{\text{ang}}A}P_{j}K_{i}\right)\right).
    \end{align}
    Now we first let $i\to\infty$ and then $j\to\infty$, and both limits may passed under the trace by the same reasons as before and we obtain in total
    \begin{align}
    &\sum_{k=1}^{d}\Tr_{L^2(\IR^{d},\IC^{r}\otimes H)}\mathcal{J}\left(\dif_{d}f,H_{B,0},\mathbf{M}_{k}\right)\nonumber\\
    =&\Tr_{L^2(\IR^{d},\IC^{r}\otimes H)}\left(M_{\Ii c\nabla_{R} A}\mathcal{J}\left(\dif_{d-1}f',H_{B,0},M_{\Ii c\nabla_{\text{ang}}A}\right)\right).
    \end{align}
    Returning to (\ref{eq:prop:leftspectralshift:1}) and applying Proposition \ref{prop:divideddifference} we have
    \begin{align}\label{eq:prop:leftspectralshift:3}
        &\Tr_{L^2(\IR^{d},H)}\tr_{\IC^{r}}\left(f\left(D_{B}^{\ast}D_{B}\right)-f\left(D_{B}D_{B}^{\ast}\right)\right)\nonumber\\
        =&(d+1)\int_{0}^{1}(1-t)^{d}\Tr_{L^2(\IR^{d},\IC^{r}\otimes H)}\left[\mathcal{J}\left(\dif_{d+1}f,H_{B,-\frac{t}{2}},-M_{\Ii c\nabla A}\right)\right.\nonumber\\
        &\left.-\mathcal{J}\left(\dif_{d+1}f,H_{B,\frac{t}{2}},M_{\Ii c\nabla A}\right)\right]\Id t\nonumber\\
        &-2\Tr_{L^2(\IR^{d},\IC^{r}\otimes H)}\left(M_{\Ii c\nabla_{R}A}\mathcal{J}\left(\dif_{d-1}f',H_{B,0},M_{\Ii c\nabla_{\text{ang}}A}\right)\right).
    \end{align}
    According to Theorem \ref{thm:spectralshiftfunctionexist}, Lemma \ref{lem:hbproperty} and Lemma \ref{lem:support} there exist unique functions $\eta_{1,t}:=\eta_{d+1,H_{B,-\frac{t}{2}},-M_{\Ii c\nabla A},-M_{\Ii c\nabla A}}$, $\eta_{2,t}:=\eta_{d+1,H_{B,\frac{t}{2}},M_{\Ii c\nabla A},M_{\Ii c\nabla A}}$ and $\eta_{3}:=\eta_{d-1,H_{B,0},M_{\Ii c\nabla_{R}A},M_{\Ii c\nabla_{\text{ang}}A}}$, all of which are elements in $\langle X\rangle^{N}L^{1}(\IR_{\geq 0})$, such that for $f\in W^{N,d+1}_{0}\cap\bigcap_{j=0}^{d+1}W^{0,j}_{0}$
    \begin{align}
        &\Tr_{L^2(\IR^{d},H)}\tr_{\IC^{r}}\left(f\left(D_{B}^{\ast}D_{B}\right)-f\left(D_{B}D_{B}^{\ast}\right)\right)\nonumber\\
        =&(d+1)\int_{0}^{1}(1-t)^{d}\int_{0}^{\infty}f^{(d+1)}(\lambda)(\eta_{1,t}(\lambda)-\eta_{2,t}(\lambda))\Id\lambda~\Id t\nonumber\\
        &+\int_{0}^{\infty}f^{(d)}(\lambda)\eta_{3}(\lambda)\Id\lambda.
    \end{align}
    The $\langle X\rangle^{N}L^{1}(\IR_{\geq 0})$-norm of the functions $\eta_{i,t}$ is uniform in $t\in[0,1]$ and depends measurably on $t$ by Lemma \ref{lem:measureable}. We thus define $\Xi_{d,A_{0},B}\in\langle X\rangle^{N+1}L^{1}(\IR_{\geq 0})$ via
    \begin{align}
        \Xi_{d,A_{0},B}(\lambda):=(d+1)\int_{0}^{1}(1-t)^{d}(\eta_{1,t}(\lambda)-\eta_{2,t}(\lambda))\Id t-\int_{0}^{\lambda}\eta_{3}(\mu)\Id\mu,~\lambda\in\IR,
    \end{align}
    which after integration by parts gives for $f\in C^{\infty}_{c}(\IR)$,
    \begin{align}\label{eq:prop:leftspectralshift:4}
        \Tr_{L^2(\IR^{d},H)}\tr_{\IC^{r}}\left(f\left(D_{B}^{\ast}D_{B}\right)-f\left(D_{B}D_{B}^{\ast}\right)\right)=\int_{0}^{\infty}f^{(d+1)}(\lambda)\Xi_{d,A_{0},B}(\lambda)\Id\lambda.
    \end{align}
    Density and continuity extends (\ref{eq:prop:leftspectralshift:4}) to $f\in W^{N,d+1}_{0}\cap\bigcap_{j=0}^{d+1}W^{0,j}_{0}$.
\end{proof}

We arrive at the principal result of this paper, which provides the connection of both spectral shift functions by a functional equation, which generalizes the result in \cite{Push} to higher dimensions. Moreover, we replace the spectral shift function $\Xi_{d,A_{0},B}$ by its derivative, which has the more convenient order $d$.

\begin{thm}\label{thm:pushfunctional}
    Assume Hypothesis \ref{hyp:smoothhyp}. Then the function $\Xi_{d,A_{0},B}$ from Proposition \ref{prop:leftspectralshift} is weakly differentiable with $\xi_{d,A_{0},B}:=-\Xi_{d,A_{0},B}'\in\langle X\rangle^{N+\frac{d}{2}}L^{1}(\IR_{\geq 0})$, such that for 
    $f\in W^{N,d}_{0}\cap W^{N,d+1}_{0}\cap\bigcap_{j=0}^{d+1}W^{0,j}_{0}$ with $f^{(d)}\in\langle X\rangle^{-N-\frac{d}{2}}L^{1}(\IR)$,
    \begin{align}
         \Tr_{L^2(\IR^{d},H)}\tr_{\IC^{r}}\left(f\left(D_{B}^{\ast}D_{B}\right)-f\left(D_{B}D_{B}^{\ast}\right)\right)=\int_{0}^{\infty}f^{(d)}(\lambda)\xi_{d,A_{0},B}(\lambda)\Id\lambda.
    \end{align}
    The function $\xi_{d,A_{0},B}$ satisfies the functional equation
    \begin{align}
        \xi_{d,A_{0},B}(\lambda)=-\frac{\left(\frac{d-1}{2}\right)!}{\pi^{\frac{d+1}{2}}(d-1)!}\int_{0}^{\lambda}(\lambda-\mu)^{\frac{d}{2}-1}\eta_{d,A_{0},B}(\mu)\Id\mu,~\lambda>0,
    \end{align}
    where $\eta_{d,A_{0},B}$ is the function obtained from Proposition \ref{prop:spectralcallias}.
\end{thm}

\begin{proof}
    Let $t>0$, and let $f_{t}$ be a Schwartz function on $\IR$ such that $f_{t}(x)=e^{-tx}$ for $x\geq 0$. In particular $f_{t}\in W^{N,d}_{0}\cap W^{N,d+1}_{0}\cap\bigcap_{j=0}^{d+1}W^{0,j}_{j}$. Then Theorem \ref{thm:principaltrace} translates to
    \begin{align}
        &\Tr_{L^2(\IR^{d},H)}\tr_{\IC^{r}}\left(f_{t}\left(D_{B}^{\ast}D_{B}\right)-f_{t}\left(D_{B}D_{B}^{\ast}\right)\right)\nonumber\\
        =&-\frac{2}{d}(4\pi)^{-\frac{d}{2}}t^{-\frac{d}{2}}\int_{\IR^{d}}\Tr_{\IC^{r}\otimes H}\left(\Ii c\nabla A_{\phi}(x)\mathcal{J}\left(\dif_{d-1}f_{t}',A_{\phi}^{2}(x),\mathbf{\Ii c\nabla A_{\phi}(x)}\right)\right)\Id x,
    \end{align}
    Both sides can be re-expressed through the spectral shift functions given in Proposition \ref{prop:spectralcallias} and Proposition \ref{prop:leftspectralshift} respectively:
    \begin{align}
        \int_{0}^{\infty}f_{t}^{(d+1)}(\lambda)\Xi_{d,A_{0},B}(\lambda)\Id\lambda=-2(4\pi)^{-\frac{d}{2}}t^{-\frac{d}{2}}\int_{0}^{\infty}f_{t}^{(d)}(\mu)\eta_{d,A_{0},B}(\mu)\Id\mu,
    \end{align}
        which gives with $\Xi:=\Xi_{d,A_{0},B}$, and $\eta:=\eta_{d,A_{0},B}$,
    \begin{align}
        &\int_{0}^{\infty}e^{-t\lambda}\Xi(\lambda)\Id\lambda=2(4\pi)^{-\frac{d}{2}}t^{-\frac{d}{2}-1}\int_{0}^{\infty}e^{-t\mu}\eta(\mu)\Id\mu\nonumber\\
        =&2(4\pi)^{-\frac{d}{2}}\Gamma\left(\frac{d}{2}+1\right)^{-1}\int_{0}^{\infty}\int_{0}^{\infty}e^{-t(s+\mu)}s^{\frac{d}{2}}\eta(\mu)\Id\mu~\Id s\nonumber\\
        =&\frac{2}{d}\pi^{-\frac{d+1}{2}}\frac{\left(\frac{d-1}{2}\right)!}{(d-1)!}\int_{0}^{\infty}e^{-t\lambda}\int_{0}^{\lambda}(\lambda-\mu)^{\frac{d}{2}}\eta(\mu)\Id\mu~\Id\lambda.
    \end{align}
    Denote for $\lambda\geq0$
    \begin{align}
        \psi(\lambda):=\Xi(\lambda)-\frac{2}{d}\pi^{-\frac{d+1}{2}}\frac{\left(\frac{d-1}{2}\right)!}{(d-1)!}\int_{0}^{\lambda}(\lambda-\mu)^{\frac{d}{2}}\eta(\mu)\Id\mu.
    \end{align}
    We claim that $\psi=0$ a.e. on $\IR_{\geq 0}$. To that end, we first note that for all $\epsilon>0$ we have $\lambda\mapsto e^{-s\lambda}\psi(\lambda)\in L^{1}(\IR_{\geq 0})$. The Stone-Weierstra{\ss} theorem implies that polynomials in $e^{-\lambda}$ are dense in $C_{0}(\IR_{\geq 0})$, and therefore we have for all $\phi\in C_{0}(\IR_{\geq 0})$ that
    \begin{align}
        \int_{0}^{\infty}\phi(\lambda)e^{-\epsilon\lambda}\psi(\lambda)\Id\lambda=0,
    \end{align}
    which implies that $\psi=0$ a.e. by the fundamental lemma of variational calculus. Thus, for a.e. $\lambda\geq 0$,
    \begin{align}
        \Xi(\lambda)=\frac{2}{d}\pi^{-\frac{d+1}{2}}\frac{\left(\frac{d-1}{2}\right)!}{(d-1)!}\int_{0}^{\lambda}(\lambda-\mu)^{\frac{d}{2}}\eta(\mu)\Id\mu,
    \end{align}
    where the right hand side is a weakly differentiable representative. Indeed let 
    \begin{align}
        \xi_{d,A_{0},B}(\lambda):=-\pi^{-\frac{d+1}{2}}\frac{\left(\frac{d-1}{2}\right)!}{(d-1)!}\int_{0}^{\lambda}(\lambda-\mu)^{\frac{d}{2}-1}\eta(\mu)\Id\mu,
    \end{align}
    which is an element in $\langle X\rangle^{N+\frac{d}{2}}L^{1}(\IR_{\geq 0})$, since $\eta\in\langle X\rangle^{N}L^{1}(\IR_{\geq 0})$. For $f\in C_{c}^{\infty}(\IR_{\geq 0})$,
    \begin{align}
        &\int_{0}^{\infty}f'(\lambda)\int_{0}^{\lambda}(\lambda-\mu)^{\frac{d}{2}}\eta(\mu)\Id\mu~\Id\lambda=\int_{0}^{\infty}\eta(\mu)\int_{\mu}^{\infty}f'(\lambda)(\lambda-\mu)^{\frac{d}{2}}\Id\lambda~\Id\mu\nonumber\\
        =&-\frac{d}{2}\int_{0}^{\infty}\eta(\mu)\int_{\mu}^{\infty}f(\lambda)(\lambda-\mu)^{\frac{d}{2}-1}\Id\lambda~\Id\mu\nonumber\\
        =&-\frac{d}{2}\int_{0}^{\infty}f(\lambda)\int_{0}^{\lambda}(\lambda-\mu)^{\frac{d}{2}-1}\eta(\mu)\Id\mu~\Id\lambda.
    \end{align}
    In particular
    \begin{align}
        \xi_{d,A_{0},B}(\lambda):=-\Xi'(\lambda),
    \end{align}
    and consequently for $f\in W^{N,d}_{0}\cap W^{N,d+1}_{0}\cap\bigcap_{j=0}^{d+1}W^{0,j}_{j}$ with\\
    $f^{(d)}\in \langle X\rangle^{-N-\frac{d}{2}}L^{1}(\IR_{\geq 0})$,
    \begin{align}
        \int_{0}^{\infty}f^{(d+1)}(\lambda)\Xi_{d,A_{0},B}(\lambda)\Id\lambda=\int_{0}^{\infty}f^{(d)}(\lambda)\xi_{d,A_{0},B}(\lambda)\Id\lambda.
    \end{align}
\end{proof}

We close this chapter with a result which shows that $\xi_{d,A_{0},B}$ is in fact more regular, which is an immediate consequence of the functional equation in Theorem \ref{thm:pushfunctional}.

\begin{cor}\label{cor:spectralshiftregular}
    Assume Hypothesis \ref{hyp:smoothhyp}, and let $\xi_{d,A_{0},B}$ be as in Theorem \ref{thm:pushfunctional}. Then $\xi_{d,A_{0},B}$ is $k=\frac{d-1}{2}$-times weakly differentiable with\\
    $\xi_{d,A_{0},B}^{(k)}\in\langle X\rangle^{N+\frac{1}{2}}L^1(\IR_{\geq 0})$, and for a.e. $\lambda>0$,
    \begin{align}\label{eq:cor:spectralshiftregular:1}
        \xi_{d,A_{0},B}^{(k)}(\lambda)=-\frac{1}{\pi}(4\pi)^{-k}\int_{0}^{\lambda}(\lambda-\mu)^{-\frac{1}{2}}\eta_{d,A_{0},B}(\mu)\Id\mu.
    \end{align}
\end{cor}

\begin{proof}
Let $\eta=\eta_{d,A_{0},B}$ on $\IR_{\geq 0}$ and $\eta\equiv 0$ on $\IR_{<0}$, $\xi=\xi_{d,A_{0},B}$ on $\IR_{\geq 0}$ and $\xi\equiv 0$ on $\IR_{<0}$. Finally denote by $X_{+}^{s}(x):=\one_{x>0}x^{s}$, for $x,s\in\IR$. By Theorem \ref{thm:pushfunctional} we may express $\xi$ via the convolution
\begin{align}
    \xi=-\frac{\left(\frac{d-1}{2}\right)!}{\pi^{\frac{d+1}{2}}(d-1)!}X_{+}^{k+\frac{1}{2}}\ast\eta,
\end{align}
which ensures that $\xi_{d,A_{0},B}$ is $k$-times weakly differentiable satisfying (\ref{eq:cor:spectralshiftregular:1}), where the right hand side integral of (\ref{eq:cor:spectralshiftregular:1}) converges for a.e. $\lambda>0$ (convolution integral). (\ref{eq:cor:spectralshiftregular:1}) also implies that $\xi_{d,A_{0},B}^{(k)}\in\langle X\rangle^{N+\frac{1}{2}}L^1(\IR_{\geq 0})$.
\end{proof}

\section{Regularized index}

In the last chapter we established the functional equation in Theorem \ref{thm:pushfunctional} for the spectral shift functions $\eta_{d,A_{0},B}$, and $\xi_{d,A_{0},B}$, which are constructed in Proposition \ref{prop:spectralcallias} and Proposition \ref{prop:leftspectralshift} respectively. Such a functional equation was shown first by A. Pushnitski in \cite{Push} in the one-dimensional case $d=1$. Following the functional equation the author is then able to prove an extension of the "Index=Spectral Flow" theorem. In this chapter we will use a similar approach to give an extension of the Callias' index theorem. The key to this approach is the regular behaviour of the spectral shift functions near zero.

More precisely we showed in Corollary \ref{cor:spectralshiftregular} that the functional equation implies that $\xi_{d,A_{0},B}$ is in general $\frac{d-1}{2}$-times differentiable. In this chapter we investigate the situation in which $\xi_{d,A_{0},B}$ is even more regular, at least from the right at $0$. We will show that such a condition gives rise to a notion of regularized index, which corresponds to the Witten index in the one-dimensional case. The idea of using spectral shift to generalize the Fredholm index to non-Fredholm situations has a tradition which goes back to \cite{GesSim}. Our first result in this chapter is a minor modification of \cite{GesSim}[Theorem 2.4] and a result in \cite{CGLS}.

\begin{prop}\label{prop:heatwitten}
    Assume Hypothesis \ref{hyp:smoothhyp}. Assume that the spectral shift function $\xi_{d,A_{0},B}$ is $(d-1)$-times weakly differentiable with $\xi_{d,A_{0},B}^{(d-1)}\in e^{tX}L^{1}(\IR)$ for some $t>0$. Furthermore assume that $\xi_{d,A_{0},B}^{(d-1)}$ has a right Lebesgue point at $0$. Then the following limit exists
    \begin{align}
        \lim_{t\to\infty}\Tr_{L^2(\IR^{d},H)}\tr_{\IC^{r}}\left(e^{-tD_{B}^{\ast}D_{B}}-e^{-tD_{B}D_{B}^{\ast}}\right)=-\xi_{d,A_{0},B}^{(d-1)}(0+).
    \end{align}
\end{prop}

\begin{proof}
    Let $t>0$, and set $\xi:=\xi_{d,A_{0},B}$. Then
    \begin{align}
        &\Tr_{L^2(\IR^{d},H)}\tr_{\IC^{r}}\left(e^{-tD_{B}^{\ast}D_{B}}-e^{-tD_{B}D_{B}^{\ast}}\right)=\int_{0}^{\infty}(-t)^{d}e^{-t\lambda}\xi(\lambda)\Id\lambda\nonumber\\
        =&\int_{0}^{\infty}(-t)e^{-t\lambda}\xi^{(d-1)}(\lambda)\Id\lambda=-t^{2}\int_{0}^{\infty}e^{-t\lambda}\int_{0}^{\lambda}\xi^{(d-1)}(\mu)\Id\mu~\Id\lambda\nonumber\\
        =&-\int_{0}^{\infty}\nu e^{-\nu}\left[\frac{t}{\nu}\int_{0}^{\frac{\nu}{t}}\xi^{(d-1)}(\mu)\Id\mu\right]\Id\nu.
    \end{align}    
    The expression in $[\cdot]$-brackets converges for $\nu>0$ to $\xi^{(d-1)}(0+)$ a $t\to\infty$, by the definition of a right Lebesgue point. Since $\int_{0}^{\infty}\nu e^{-\nu}\Id\nu=1$, we may assume without loss in the following that $\xi_{d,A_{0},B}^{(d-1)}(0+)=0$.
    
    In particular, there are constants $\epsilon>0$ and $C>0$, such that for $\nu<\epsilon t$,
    \begin{align}
        \left|\frac{t}{\nu}\int_{0}^{\frac{\nu}{t}}\xi^{(d-1)}(\mu)\Id\mu\right|\leq C.
    \end{align}
    On the other hand there is $t_{0}$, such that for $t\geq t_{0}$,
    \begin{align}
        \left|\frac{t}{\nu}\int_{0}^{\frac{\nu}{t}}\xi^{(d-1)}(\mu)\Id\mu\right|\leq Ce^{\frac{\nu}{2}}.
    \end{align}
    So there is a constant $C'$, such that for $t\geq t_{0}$, $\nu>0$,
    \begin{align}
        \nu e^{-t\nu}\left|\frac{t}{\nu}\int_{0}^{\frac{\nu}{t}}\xi^{(d-1)}(\mu)\Id\mu\right|\leq C\nu e^{-\nu}\one_{\nu<\epsilon t}+C\epsilon^{-1}\nu e^{-\frac{\nu}{2}}\one_{\nu\geq\epsilon t}\leq C'\nu e^{-\frac{\nu}{2}}.
    \end{align}
    The claim then follows by application of the dominated convergence theorem as $t\to\infty$.
\end{proof}

In view of the definition of semi-group regularized Witten index in \cite{GesSim}, and the proposition above, we suggest the following definition of partial Witten index.

\begin{defn}\label{defn:partialwitten}
    Assume Hypothesis \ref{hyp:smoothhyp}. Assume that the spectral shift function $\xi_{d,A_{0},B}$ is $(d-1)$-times weakly differentiable with $\xi_{d,A_{0},B}^{(d-1)}\in e^{tX}L^{1}(\IR)$ for some $t>0$. Furthermore assume that $\xi_{d,A_{0},B}^{(d-1)}$ admits a right Lebesgue point at $0$. Then define the \textit{partial Witten index}
    \begin{align}
        \ind_{W}D_{B}:=-\xi_{d,A_{0},B}^{(d-1)}(0+).
    \end{align}
\end{defn}

The next Lemma justifies the word "index" for the quantity $\ind_W D_{B}$, because in the Fredholm case it agrees with the Fredholm index of $D_{B}$.

\begin{lem}\label{lem:fredholmwitten}
    Assume Hypothesis \ref{hyp:smoothhyp}, and assume that $D_{B}$ is Fredholm with index $\ind D_{B}$. Then $D_{B}$ admits a partial Witten index and
    \begin{align}
        \ind_{W}D_{B}=\ind D_{B}.
    \end{align}
\end{lem}

\begin{proof}
    Let $\delta>0$, such that $[0,\delta)\cap\left(\sigma(D_{B}^{\ast}D_{B})\cup\sigma(D_{B}D_{B}^{\ast})\right)\subseteq\{0\}$. Let
    \begin{align}
        \xi(\lambda):=\begin{cases}\xi_{d,A_{0},B}(\lambda),&\lambda\geq 0,\\
        0,&\lambda<0,
    \end{cases}
    \end{align}
    and for $f\in C^{\infty}_{c}\left((-\delta,\delta)\right)$ define the distribution
    \begin{align}
        T(f):=\int_{-\delta}^{\delta}f(\lambda)\xi(\lambda)\Id\lambda.
    \end{align}
    Then, for all $f\in C^{\infty}_{c}\left((-\delta,\delta)\right)$,
    \begin{align}
        &f(0)\ind D_{B}=f(0)\left(\dim\ker D_{B}^{\ast}D_{B}-\dim\ker D_{B}D_{B}^{\ast}\right)\nonumber\\
        =&\Tr_{L^2(\IR^{d},H)}\tr_{\IC^{r}}\left(f(D_{B}^{\ast}D_{B})-f(D_{B}D_{B}^{\ast})\right)\nonumber\\
        =&\int_{0}^{\infty}f^{(d)}(\lambda)\xi_{d,A_{},B}(\lambda)\Id\lambda=T\left(f^{(d)}\right)=-T^{(d)}(f),
    \end{align}
    which implies that there exists a polynomial $p$ of degree $d-1$ such that for a.e. $\lambda\in(-\delta,\delta)$,
    \begin{align}
        \xi(\lambda)=p(\lambda)-\ind D_{B}\cdot\begin{cases}
            \frac{\lambda^{d-1}}{(d-1)!},&\lambda\geq 0,\\
            0,&\lambda<0.          
        \end{cases}
    \end{align}
    Since $\xi\equiv 0$ on $\IR_{<0}$, it follows that $p\equiv 0$. In particular $\xi$ is $(d-1)$-times weakly differentiable with
    $\xi^{(d-1)}(\lambda)=-\ind D_{B}$ for a.e. $0<\lambda<\delta$, which proves the claim.
\end{proof}

The functional equation in Theorem \ref{thm:pushfunctional} allows us to impose conditions on the spectral shift function $\eta_{d,A_{0},B}$, which ensure that $\ind_{W}D_{B}$ exists. Moreover we obtain an index formula in this case. This is the second main result of this paper with which we close this chapter. It provides the appropriate analogue for "$A$ needs to be invertible outside a compact region" implied by the Fredholmness condition in the classical Callias' theorem translated to our in general non-Fredholm setup.

\begin{thm}\label{thm:indexformula}
    Assume Hypothesis \ref{hyp:smoothhyp}. Assume that the spectral shift function $\eta_{d,A_{0},B}$ is $\frac{d-1}{2}$-times weakly differentiable with $\eta_{d,A_{0},B}^{(j)}(0)=0$ for all $0\leq j\leq\frac{d-3}{2}$. Then $\xi_{d,A_{0},B}$ is $(d-1)$-times weakly differentiable such that for a.e. $\lambda>0$,
    \begin{align}\label{eq:thm:indexformula:1}
        \xi_{d,A_{0},B}^{(d-1)}(\lambda)=-\frac{1}{\pi}(4\pi)^{-\frac{d-1}{2}}\int_{0}^{\lambda}(\lambda-\mu)^{-\frac{1}{2}}\eta_{d,A_{0},B}^{(\frac{d-1}{2})}(\mu)\Id\mu.
    \end{align}
    If additionally $\eta_{d,A_{0},B}^{(\frac{d-1}{2})}\in e^{tX}L^{1}(\IR)$ for some $t>0$, and the function $u\mapsto u\eta_{d,A_{0},B}^{(\frac{d-1}{2})}(u^{2})$ admits a right Lebesgue point at $0$ with value $L_{d,A_{0},B}$, i.e.
    \begin{align}
        \lim_{h\searrow 0}\int_{0}^{h}\left(u\eta_{d,A_{0},B}^{(\frac{d-1}{2})}(u^{2})-L_{d,A_{0},B}\right)\Id u=0,
    \end{align}
    then also $\xi_{d,A_{0},B}^{(d-1)}$ admits a right Lebesgue point at $0$, and
    \begin{align}\label{eq:thm:indexformula:2}
        \ind_W D_{B}=-\xi_{d,A_{0},B}^{(d-1)}(0+)=(4\pi)^{-\frac{d-1}{2}}L_{d,A_{0},B}.
    \end{align}
\end{thm}

\begin{proof}
Let $k:=\frac{d-1}{2}$, $\eta=\eta_{d,A_{0},B}$ on $\IR_{\geq 0}$ and $\eta\equiv 0$ on $\IR_{<0}$, $\xi=\xi_{d,A_{0},B}$ on $\IR_{\geq 0}$ and $\xi\equiv 0$ on $\IR_{<0}$. Finally denote by $X_{+}^{s}(x):=\one_{x>0}x^{s}$, for $x,s\in\IR$. By Corollary \ref{cor:spectralshiftregular} we may express $\xi^{(k)}$ via the convolution
\begin{align}
    \xi^{(k)}=-\frac{1}{\pi}(4\pi)^{-k}X_{+}^{-\frac{1}{2}}\ast\eta.
\end{align}
The boundary conditions on $\eta_{d,A_{0},B}$ ensure that $\eta$ is also $k$-times weakly differentiable. Thus $\xi$ is also $(d-1)$-times weakly differentiable and for a.e. $\lambda>0$,
\begin{align}
    \xi_{d,A_{0},B}^{(d-1)}(\lambda)=-\frac{1}{\pi}(4\pi)^{-k}\left(X_{+}^{-\frac{1}{2}}\ast\eta^{(k)}\right)(\lambda),
\end{align}
which gives (\ref{eq:thm:indexformula:1}), which implies for $h>0$, and $f(u):=\frac{1}{\pi}(4\pi)^{-\frac{d-1}{2}}u\eta_{d,A_{0},B}^{(\frac{d-1}{2})}(u^2)$,
\begin{align}
    &\frac{1}{h}\int_{0}^{h}\xi_{d,A_{0},B}^{(d-1)}(\lambda)\Id\lambda=-\frac{1}{h}\int_{0}^{h}\int_{0}^{\lambda}(\lambda-\mu)^{-\frac{1}{2}}\mu^{-\frac{1}{2}}f(\mu^{\frac{1}{2}})\Id\mu~\Id\lambda\nonumber\\
    =&-\frac{1}{h}\int_{0}^{h}\int_{0}^{\lambda^{\frac{1}{2}}}(\lambda-u^{2})^{-\frac{1}{2}}f(u)\Id u~\Id\lambda=-\frac{2}{h}\int_{0}^{h^{\frac{1}{2}}}(h-u^{2})^{\frac{1}{2}}f(u)\Id u.
\end{align}
Since $\frac{2}{h}\int_{0}^{h^{\frac{1}{2}}}(h-u^{2})^{\frac{1}{2}}\Id u=\pi$, without loss we may assume that $f(0+)=0$. Then
\begin{align}
    \left|\frac{1}{h}\int_{0}^{h}\xi_{d,A_{0},B}^{(d-1)}(\lambda)\Id\lambda\right|\leq\frac{2}{h^{\frac{1}{2}}}\int_{0}^{h^{\frac{1}{2}}}\left|f(u)\right|\Id u\xrightarrow{h\searrow 0}0,
\end{align}
which shows (\ref{eq:thm:indexformula:2}).
\end{proof}

\section{Example and comparison with results in one dimension}

    Before we present a concrete example to which we apply the theory developed in the previous chapter, it is didactically sensible to compare our results, which cover the case $d\geq 3$ odd with known results for the case $d=1$. Indeed, although beyond the scope of this paper, the formula presented in Theorem \ref{thm:indexformula} is compatible with the formula shown in \cite{CGLS} in the case $d=1$. To see that let us pretend that our definition of the spectral shift functions $\eta=\eta_{d,A_{0},B}$ and $\xi=\xi_{d,A_{0},B}$ extends to $d=1$. Let us also assume that $A_{\pm}:=\lim_{x\to\pm\infty}A(x)$ exist strongly on $\dom A_{0}$ with $A_{-}=A_{0}$, and that the Krein spectral shift function $\kappa=\kappa_{A_{+},A_{-}}\in L^{1}_{loc}(\IR)$ of the pair $A_{\pm}$ exists. Then let $A_{\phi}(x):=A_{-}+(1-\phi(x))B(x)$ for $\phi\in C_{c}^{\infty}(\IR)$ with $\phi\equiv 1$ near $0$. In particular $\lim_{x\to\pm\infty}A_{\phi}(x)=A_{\pm}$ on $\dom A_{-}$. We make the following heuristic observations.

    \begin{heur}
        Let $f\in C_{c}^{\infty}((0,\infty))$, and define the bounded smooth function $F:\IR\rightarrow\IC$ by
        \begin{align}
            F(x):=\operatorname{sgn}(x)\int_{0}^{x^{2}}f(y)\frac{\Id y}{2\sqrt{y}},
        \end{align}
        which satisfies $F'(x)=f(x^{2}).$
        Then
        \begin{align}
            &\int_{0}^{\infty}f(x)\frac{\kappa(\sqrt{x})+\kappa(-\sqrt{x})}{2\sqrt{x}}\Id x=\int_{\IR}f(x^{2})\kappa(x)\Id x=\Tr_{H}\left(F(A_{+})-F(A_{-})\right)\nonumber\\
            =&\int_{\IR}\Tr_{H}\p_{x}F(A_{\phi}(x))\Id x=\int_{\IR}\Tr_{H}\left(A_{\phi}'(x)F'(A_{\phi}(x))\right)\Id x\nonumber\\
            =&\int_{\IR}\Tr_{H}\left(A_{\phi}'(x)f(A_{\phi}^{2}(x))\right)\Id x=\int_{0}^{\infty}f(x)\eta(x)\Id x,
        \end{align}
        where in the penultimate line we used a well-known result on differentiation under the trace (we refer to \cite{Pel2} for a rigorous discussion). This shows that for a.e. $x>0$,
    \begin{align}
        \eta(x)=\frac{\kappa(\sqrt{x})+\kappa(-\sqrt{x})}{2\sqrt{x}}.
    \end{align}
    For $g\in C_{c}^{\infty}((0,\infty))$, and $G(x):=\int_{0}^{\infty}g(y)\Id y$ we have
    \begin{align}
        \int_{0}^{\infty}g(x)\xi(x)\Id x=\Tr_{L^{2}(\IR,H)}\left(G(D_{B}^{\ast}D_{B})-G(D_{B}D_{B}^{\ast})\right)=\int_{0}^{\infty}g(x)K(x)\Id x,
    \end{align}
    where $K=K_{D_{B}^{\ast}D_{B},D_{B}D_{B}^{\ast}}$ is the Krein spectral shift function of the pair\\
    $(D_{B}^{\ast}D_{B},D_{B}D_{B}^{\ast})$. It follows that $\xi=K$ holds a.e..

    The functional equation of Theorem \ref{thm:pushfunctional} then gives for a.e. $\lambda>0$,
    \begin{align}
        K(\lambda)=-\frac{1}{\pi}\int_{0}^{\lambda}(\lambda-\mu)^{-\frac{1}{2}}\frac{\kappa(\sqrt{\mu})+\kappa(-\sqrt{\mu})}{2\sqrt{\mu}}\Id\mu=-\frac{1}{\pi}\int_{-\sqrt{\lambda}}^{\sqrt{\lambda}}(\lambda-\mu^{2})^{-\frac{1}{2}}\kappa(\mu)\Id\mu
    \end{align}
    which is the principal functional equation of spectral shift functions found in \cite{Push},\cite{GLMST} and \cite{CGLS}.

    Moreover we have for $u>0$
    \begin{align}
        u\eta(u^{2})=\frac{\kappa(u)+\kappa(-u)}{2},
    \end{align}
    which in regard of Theorem \ref{thm:indexformula} implies that $D_{B}$ admits a (partial) Witten index $\ind_{W}D_{B}$ if
    $\frac{\kappa(u)+\kappa(-u)}{2}$ admits a (right) Lebesgue point at $0$ with value $L$, and that in that case we have the index formula
    \begin{align}
        \ind_{W}D_{B}=L.
    \end{align}
    This formula is a slight generalization of the Witten index formulas shown in \cite{Push},\cite{GLMST}, and \cite{CGLS} (of which the last paper listed gives the most general version), since $\kappa$ admitting a left and right Lebesgue point at $0$ implies that $u\mapsto\frac{\kappa(u)+\kappa(-u)}{2}$ admits a Lebesgue point at $0$.
    \end{heur}

    To illustrate our results from the previous chapter, let us conclude this paper with a non-trivial example. We will consider the so-called $(d+1)$-massless Dirac-Schr\"odinger operator, of which the trace formula in Theorem \ref{thm:principaltrace} was discussed in \cite{Furst3} and we will determine both spectral shift functions $\xi_{d,A_{0},B}$ and $\eta_{d,A_{0},B}$ in this case. It should be noted that $(1+1)$-massless Dirac-Schr\"odinger operators where first discussed in \cite{BolGes}, and provided the first examples of differential operators which attain any real number as Witten index. Such a surjectivity result also holds for the partial Witten index in higher dimensions, which has been shown in \cite{Furst3}.

    \begin{defn}\label{defn:diracschroedinger}
    Let $d\geq 3$ be odd, and $G$ a separable complex Hilbert space. Let $V\in C_{b}^{4}\left(\IR^{d}\times\IR,B_{sa}\left(G\right)\right)$, where $B_{sa}\left(G\right)$ is equipped with the strong operator topology. Furthermore assume that there exists $\epsilon>0$ such that for $i\in\left\{1,\ldots,d\right\}$, $k\in\left\{0,1,2\right\}$, and $\gamma\in\IN^{d}$ with $0\leq|\gamma|\leq 2$,
		\begin{align}
			&\left\|y\mapsto\p_{x^{i}}\p_{y}^{k}V\left(x,y\right)\right\|_{L^{d}\left(\IR,S^{d}\left(G\right)\right)}+\left\|y\mapsto\p_{x}^{\gamma}\p_{x^{i}}V\left(x,y\right)\right\|_{L^{d}\left(\IR,S^{d}\left(G\right)\right)}\nonumber\\
            =&O\left(\left|x\right|^{-1}\right),\ \left|x\right|\to\infty,\nonumber\\
			&\left\|y\mapsto\p_{R}\p_{y}^{k}V\left(x,y\right)\right\|_{L^{d-\epsilon}\left(\IR,S^{d-\epsilon}\left(G\right)\right)}+\left\|y\mapsto\p^{\gamma}_{x}\p_{R}V\left(x,y\right)\right\|_{L^{d-\epsilon}\left(\IR,S^{d-\epsilon}\left(G\right)\right)}\nonumber\\
            =&O\left(\left|x\right|^{-1-\epsilon}\right),\ \left|x\right|\to\infty,
		\end{align}
		where $\p_{R}=\sum_{i=1}^{d}\frac{x^{i}}{\left|x\right|}\p_{x^{i}}$ is the radial derivative. We also assume that for any $\widehat{x}\in S_{1}\left(0\right)$, $x_{0}\in\IR^{d}$, and $\gamma\in\IN_{0}^{d}$ with $\left|\gamma\right|\leq 1$, we have
		\begin{align}
			\lim_{R\to\infty}R^{\left|\gamma\right|}\left\|\left(\p^{\gamma}_{x}V\right)\left(x_{0}+R\widehat{x},\cdot\right)-\left(\p^{\gamma}_{x}V\right)\left(R\widehat{x},\cdot\right)\right\|_{L^{\infty}\left(\IR,B\left(G\right)\right)}=0.
		\end{align}
        Define the \textit{$(d+1)$- massless Dirac-Schr\"odinger operator with potential $V$} for $f\in W^{1,2}\left(\IR^{d+1},\IC^{r}\otimes G\right)$, $x\in\IR^{d}$, and $y\in\IR$ by
        \begin{align}
		\left(D_{V}f\right)\left(x,y\right):=\Ii\sum_{j=1}^{d}c^{j}\p_{x^{j}}f\left(x,y\right)+\Ii\p_{y}f\left(x,y\right)+V\left(x,y\right)f\left(x,y\right),
	\end{align}
    where $r=2^{\frac{d-1}{2}}$ is the rank of the Clifford matrices $(c^{j})_{j=1}^{d}$.
    \end{defn}

    For such an operator we have the following result.

    \begin{thm}\label{thm:concreteexample}
    Assume that $V$ and $D_{V}$ are as in Definition \ref{defn:diracschroedinger}. Then Hypothesis \ref{hyp:smoothhyp} holds with $H=L^2(\IR,G)$, $A_{0}=\Ii\p_{y}$, $B(x)=M_{V(x,\cdot)}$ for $\alpha>1$ small enough such that $N=1$. Let the family of unitaries $U^{V}$ be given by $U^{V}\left(x\right):=\lim_{n\to\infty}U^{V\left(x,\cdot\right)}\left(n,-n\right)$, $x\in\IR^{d}$, where $U^{T}\left(y_{1},y_{2}\right)$, $y_{1},y_{2}\in\IR$, for a given $T:\IR\rightarrow B_{sa}\left(H\right)$ is the (unique) evolution system of the equation
		\begin{align}
			u'\left(y\right)=\Ii T\left(y\right) u\left(y\right),\ y\in\IR,
		\end{align}
		i.e. for $y,y_{1},y_{2},y_{3}\in\IR$,
		\begin{align}
			\p_{y_{1}}U^{T}\left(y_{1},y_{2}\right)=&\Ii T\left(y_{1}\right)U^{T}\left(y_{1},y_{2}\right),\nonumber\\
			\p_{y_{2}}U^{T}\left(y_{1},y_{2}\right)=&-\Ii U^{T}\left(y_{1},y_{2}\right)T\left(y_{2}\right),\nonumber\\
			U^{T}\left(y_{1},y_{2}\right)U^{T}\left(y_{2},y_{3}\right)=&U^{T}\left(y_{1},y_{3}\right),\nonumber\\
			U^{T}\left(y,y\right)=&1_{G}.
		\end{align}
        Then $\eta_{d,A_{0},B}\in\langle X\rangle L^{1}(\IR_{\geq 0})$, and $\xi_{d,A_{0},B}\in\langle X\rangle^{\frac{d}{2}+1}L^{1}(\IR_{\geq 0})$ exist with
        \begin{align}
            \eta_{d,A_{0},B}(\mu)=&\int_{\IR^{d}}\Omega_{d}(\mu,z)\ind_{V}(z)\Id z,\nonumber\\
            \xi_{d,A_{0},B}(\lambda)=&\int_{\IR^{d}}\Sigma_{d}(\lambda,z)\ind_{V}(z)\Id z,
        \end{align}
        where for $\lambda,\mu>0$, $z\in\IR^{d}$,
        \begin{align}
            \ind_{V}(z):=&\frac{2}{d}(4\pi)^{-\frac{d}{2}}\left(2\Ii\right)^{\frac{d-1}{2}}\int_{x\in\IR^{d}}\Tr_{G}\left(\prod_{j=1}^{d}(\Id_{x}V)(x,z_{j-1})U^{V(x,\cdot)}(z_{j-1},z_{j})\right),
            \end{align}
            \begin{align}
            \Omega_{d}(\mu,z):=&\frac{1}{2}\int_{s\in\Delta_{d-1}}\left(\prod_{j=1}^{d}s_{j}^{-\frac{1}{2}}\right)\left(\frac{\mu}{a(s,z)}\right)^{\frac{d}{4}-\frac{1}{2}}J_{\frac{d}{2}-1}\left(2\sqrt{a(s,z)\mu}\right)\Id s,\nonumber\\
            \Sigma_{d}(\lambda,z):=&-(4\pi)^{-\frac{d}{2}}\int_{s\in\Delta_{d-1}}\left(\prod_{j=1}^{d}s_{j}^{-\frac{1}{2}}\right)\left(\frac{\lambda}{a(s,z)}\right)^{\frac{d-1}{2}}J_{d-1}\left(2\sqrt{a(s,z)\lambda}\right)\Id s,\nonumber\\
            a(s,z):=&\sum_{j=1}^{d}\frac{(z_{j-1}-z_{j})^{2}}{4s_{j}},~ s\in\Delta_{d-1},
        \end{align}
        with convention $z_{0}:=z_{d}$, and where $J_{\nu}$ denote Bessel functions of the first kind. Moreover $\xi_{d,A_{0},B}$ is $(d-1)$-times differentiable and $\xi_{d,A_{0},B}^{(d-1)}$ continuous from the right at $0$. In particular $D_{V}$ admits a partial Witten index and
        \begin{align}
            \ind_{W}D_{V}&=-\xi^{(d-1)}(0+)=\left(2\pi\Ii\right)^{-\frac{d+1}{2}}\frac{\left(\frac{d-1}{2}\right)!}{d!}\int_{\IR^{d}}\Tr_{G}\left(\left(U^{V}\right)^{-1}\Id U^{V}\right)^{\wedge d},
        \end{align}
        where $\wedge$ denotes the exterior product.
    \end{thm}

    \begin{proof}
        In \cite{Furst3}[Lemma 2.5] it is shown that under the made assumptions, Hypothesis \ref{hyp:calliashyp} holds for $\alpha>1$ arbitrarily small such that one may choose $N=1$. Repeating the proof of \cite{Furst3}[Lemma 2.5] almost ad verbatim, the stronger requirement of Hypothesis \ref{hyp:smoothhyp} is also satisfied under the stronger requirements made in Definition \ref{defn:diracschroedinger}.
        In \cite{Furst3}[Theorem 3.5] it is shown that Theorem \ref{thm:principaltrace} holds with $D_{B}=D_{V}$, and that we may even choose $\phi\equiv 0$. In the proof of \cite{Furst3}[Theorem 3.5] it is moreover shown that for $A\left(x\right)=\Ii\p_{y}+M_{V\left(x,\cdot\right)}$,
        \begin{align}\label{eq:thm:concreteexample:1}
    &\frac{2}{d}(4\pi)^{-\frac{d}{2}}t^{\frac{d}{2}}\int_{\IR^{d}}\int_{s\in\Delta_{d-1}}\Tr_{L^2(\IR,\IC^{r}\otimes G)}\left(\prod_{j=0}^{d-1}(\Ii c\nabla A)(x)e^{-ts_{j}A^{2}(x)}\right)\Id s~\Id x\nonumber\\
    &=\int_{\IR^{d}}\omega_{d}(t,z)\ind_{V}(z)\Id z.
        \end{align}
        We re-express the left hand side of (\ref{eq:thm:concreteexample:1}) via the spectral shift function $\eta=\eta_{d,A_{0},B}$, and thus obtain for its Laplace transform $\mathcal{L}(\eta)$,
        \begin{align}\label{eq:thm:concreteexample:2}
            \mathcal{L}(\eta)(t)&=\frac{1}{2}(4\pi)^{\frac{d}{2}}t^{-\frac{d}{2}}\int_{\IR^{d}}\omega_{d}(t,z)\ind_{V}(z)\Id z\nonumber\\
            &=\frac{1}{2}\int_{\IR^{d}}\int_{s\in\Delta_{d-1}}\left(\prod_{j=1}^{d}s_{j}^{-\frac{1}{2}}\right)\left[t^{-\frac{d}{2}}\exp\left(-\frac{a(s,z)}{t}\right)\right]\Id s\ind_{V}(z)\Id z
        \end{align}
        A well-known result about Bessel functions (Schl\"afli's integral \cite{Wat}[§6.2 (1)] combined with Laplace transform inversion) is
        \begin{align}\label{eq:thm:concreteexample:3}
            \mathcal{L}\left(\mu\mapsto\left(\frac{\mu}{a}\right)^{\frac{d}{4}-\frac{1}{2}}J_{\frac{d}{2}-1}(2\sqrt{a\mu})\right)(t)=t^{-\frac{d}{2}}\exp\left(-\frac{a}{t}\right),~ a,t>0,
        \end{align}
        which allows us to rewrite the term in brackets $[,]$ in (\ref{eq:thm:concreteexample:2}) and apply Laplace transform inversion to obtain for $\mu>0$,
        \begin{align}
            \eta(\mu)=&\int_{\IR^{d}}\Omega_{d}(\mu,z)\ind_{V}(z)\Id z.
        \end{align}
        To obtain a formula for $\xi=\xi_{d,A_{0},B}$, we could use the functional equation from Theorem \ref{thm:pushfunctional}, however it is simpler to use Laplace transform inversion again. Indeed, (\ref{eq:thm:concreteexample:1}) together with Theorem \ref{thm:principaltrace} implies
        \begin{align}
            &\mathcal{L}(\xi)(t)=-t^{-d}\int_{\IR^{d}}\omega_{d}(t,z)\ind_{V}(z)\Id z\nonumber\\
            =&-(4\pi)^{-\frac{d}{2}}\int_{\IR^{d}}\int_{s\in\Delta_{d-1}}\left(\prod_{j=1}^{d}s_{j}^{-\frac{1}{2}}\right)\left[t^{-d}\exp\left(-\frac{a(s,z)}{t}\right)\right]\Id s\ind_{V}(z)\Id z,
        \end{align}
        which by (\ref{eq:thm:concreteexample:3}) gives
        \begin{align}
            \xi(\lambda)=&\int_{\IR^{d}}\Sigma_{d}(\lambda,z)\ind_{V}(z)\Id z.
        \end{align}
        The function $\xi$ is $(d-1)$-times differentiable with
        \begin{align}
            \xi^{(d-1)}(\lambda)=&\int_{\IR^{d}}\p_{\lambda}^{d-1}\Sigma_{d}(\lambda,z)\ind_{V}(z)\Id z,\nonumber\\
            \p_{\lambda}^{d-1}\Sigma_{d}(\lambda,z)=&-(4\pi)^{-\frac{d}{2}}\int_{s\in\Delta_{d-1}}\left(\prod_{j=1}^{d}s_{j}^{-\frac{1}{2}}\right)J_{0}\left(2\sqrt{a(s,z)\lambda}\right)\Id s.
        \end{align}
        In particular, $\xi^{(d-1)}$ is right continuous at $0$ with value
        \begin{align}
            \ind_{W}D_{V}&=-\xi^{(d-1)}(0+)=(4\pi)^{-\frac{d}{2}}\left(\int_{s\in\Delta_{d-1}}\prod_{j=1}^{d}s_{j}^{-\frac{1}{2}}\Id s\right)\left(\int_{\IR^{d}}\ind_{V}(z)\Id z\right)\nonumber\\
            &=\left(2\pi\Ii\right)^{-\frac{d+1}{2}}\frac{\left(\frac{d-1}{2}\right)!}{d!}\int_{\IR^{d}}\Tr_{G}\left(\left(U^{V}\right)^{-1}\Id U^{V}\right)^{\wedge d},
        \end{align}
        where we used that
        \begin{align}
			\int_{s\in\Delta_{d-1}}\prod_{j=1}^{d}\left(4\pi s_{j}\right)^{-\frac{1}{2}}\Id u=\left(4\pi\right)^{-\frac{1}{2}}\frac{\left(\frac{d-1}{2}\right)!}{\left(d-1\right)!},
		\end{align}
        and according to the proof of \cite{Furst3}[Theorem 3.5],
        \begin{align}
            \int_{\IR^{d}}\ind_{V}(z)\Id z=\frac{2}{d}(4\pi)^{-\frac{d}{2}}\left(2\Ii\right)^{\frac{d-1}{2}}\Ii^{-d}\int_{x\in\IR^{d}}\Tr_{G}\left(\left(U^{V}\left(x\right)\right)^{-1}\left(\Id U^{V}\right)\left(x\right)\right)^{\wedge d}.
        \end{align}
    \end{proof}

    \begin{rem}
        The Witten index formula in \ref{thm:concreteexample} is already a result in \cite{Furst3}, which was shown by calculating the trace limit in Proposition \ref{prop:heatwitten} directly. Here we go beyond that result, because we also calculated the complete spectral shift functions $\xi_{d,A_{0},B}$ and $\eta_{d,A_{0},B}$ in Theorem \ref{thm:concreteexample}.
    \end{rem}

\section{Acknowledgements}

I would like to thank Matthias Lesch for his advice, and Jilly Kevo for proofreading the texts in this paper.

\bibliographystyle{amsplain}
\bibliography{References}

\bigskip
\noindent
Oliver F\"{u}rst\\
\href{mailto:ofuerst@math.uni-bonn.de}{\Letter ~ofuerst@math.uni-bonn.de}
\\
\noindent
Mathematisches Institut\\
Universit\"{a}t Bonn\\
Endenicher Allee 60\\
53115 Bonn\\
Germany

\end{document}